\documentclass{amsart}
\usepackage{euler, eufrak, amssymb, amsmath}
\usepackage[all]{xy}

\usepackage{hyperref}

\theoremstyle{plain}
\newtheorem{theorem}{Theorem}
\newtheorem{proposition}{Proposition}[subsection]
\newtheorem{lemma}[proposition]{Lemma}

\theoremstyle{remark}

\theoremstyle{definition}

\newcommand{\omicron}{o}

\let\f\varphi

\def\lra{\longrightarrow}
\def\egal{\ar@{=}}
\def\surj{\ar@{->>}}

\def\E{\mathcal E}
\def\F{\mathcal F}
\def\G{\mathcal G}
\def\L{\mathcal L}
\def\O{\mathcal O}
\def\X{\mathcal X}
\def\Y{\mathcal Y}
\def\Z{\mathcal Z}

\def\Coker{{\mathcal Coker}}

\def\AA{\mathbb A}
\def\CC{\mathbb C}
\def\NN{\mathbb N}
\def\PP{\mathbb P}
\def\QQ{\mathbb Q}
\def\WW{\mathbb W}
\def\ZZ{\mathbb Z}

\def\D{{\scriptscriptstyle \operatorname{D}}}
\def\TR{{\scriptscriptstyle \operatorname{T}}}
\def\vir{{\scriptstyle \operatorname{vir}}}
\def\ss{{\scriptstyle \operatorname{ss}}}

\def\A{\operatorname{A}}
\def\H{\operatorname{H}}
\def\M{\operatorname{M}}
\def\N{\operatorname{N}}
\def\Sym{\operatorname{S}}
\def\s{\operatorname{s}}
\def\T{\operatorname{T}}

\def\Hom{\operatorname{Hom}}
\def\Aut{\operatorname{Aut}}
\def\End{\operatorname{End}}
\def\Ext{\operatorname{Ext}}
\def\GL{\operatorname{GL}}

\def\rank{\operatorname{rank}}
\def\spann{\operatorname{span}}
\def\Stab{\operatorname{Stab}}
\def\Hilb{\operatorname{Hilb}}
\def\type{\operatorname{type}}
\def\red{\operatorname{red}}

\def\tensor{\otimes}
\def\isom{\simeq}

\def\ba{\begin{array}}
\def\ea{\end{array}}

\begin{document}

\title[The homology groups of certain moduli spaces of plane sheaves]
{The homology groups of certain moduli spaces of plane sheaves}

\author{Mario Maican}
\address{Institute of Mathematics of the Romanian Academy,
Calea Grivitei 21, Bucharest 010702, Romania}
\email{mario.maican@imar.ro}


\begin{abstract}
Using the Bia{\l}ynicki-Birula method, we determine the additive structure of the integral homology
groups of the moduli spaces of semi-stable sheaves on the projective plane
having rank and Chern classes $(5, 1, 4)$, $(7, 2, 6)$, respectively, $(0, 5, 19)$.
We compute the Hodge numbers of these moduli spaces.
\end{abstract}

\maketitle

\section{Introduction} 
\label{section_1}

Let $r > 0$, $c_1$, $c_2$ be integers.
Let $\M(r, c_1, c_2)$ be the moduli space of Gieseker semi-stable
sheaves on $\PP^2 = \PP^2(\CC)$ of rank $r$ and Chern classes $c_1$, $c_2$.
Let
\[
\Delta(r, c_1, c_2) = \frac{1}{r} \bigg( c_2 - \bigg( 1 - \frac{1}{r} \bigg) \frac{c_1^2}{2} \bigg)
\quad \text{and} \quad \mu = \frac{c_1}{r}
\]
be the discriminant and slope of a sheaf giving a point in $\M(r, c_1, c_2)$.
The moduli spaces $\M(r, c_1, c_2)$ of dimension zero consist of a point,
the isomorphism class of a semi-exceptional bundle, (see \cite[Section 16.1]{lepotier}).
According to \cite{drezet_reine}, there exists a unique function $\delta \colon \QQ \to \QQ$
such that for all $r$, $c_1$, $c_2$, $\M(r, c_1, c_2)$ has positive dimension
if and only if $\Delta(r, c_1, c_2) \ge \delta (\mu)$.
(The function $\delta$ is positive and periodic of period $1$, see \cite[Section 16.4]{lepotier}.)
In \cite{drezet_reine}, $\M(r, c_1, c_2)$ is said to have \emph{height zero} if $\Delta = \delta(\mu)$.
It was proved in \cite{drezet_reine} that the moduli spaces of height zero are isomorphic to moduli spaces
of semi-stable Kronecker modules.

Let $q$, $m$, $n$ be positive integers.
The group $\GL(m,\CC) \times \GL(n,\CC)$ acts by conjugation on the vector space
$\Hom(\CC^m \tensor \CC^q, \CC^n)$, whose elements are called \emph{Kronecker modules}.
The subset of semi-stable Kronecker modules admits a good quotient, denoted $\N(q,m,n)$,
which, according to \cite{drezet_reine}, is an irreducible projective variety
of dimension $qmn-m^2-n^2+1$, if non-empty. It is smooth at points given by stable
Kronecker modules so, in particular, $\N(q, m, n)$ is smooth if $m$ and $n$ are coprime.
The main result of \cite{drezet_reine} states that if $\M(r, c_1, c_2)$ has height zero,
then there are an exceptional bundle $E$ and integers $m$, $n$ such that we have
a canonical isomorphism
\[
\M(r, c_1, c_2) \isom \N(3 \rank(E), m, n).
\]
Let $x_E$ denote the smallest real solution to the equation $x^2 - 3 \rank(E) x + 1 = 0$.
Let $\mu(E)$ denote the slope of $E$. Then
\[
\mu(E) - \frac{x_E}{\rank(E)} < \mu < \mu(E) + \frac{x_E}{\rank(E)}.
\]
Composing with the isomorphism $\N(q, m, n) \isom \N(q, n, qn-m)$ of \cite{drezet_reine},
we obtain a new moduli space of height zero, denoted $\Lambda^{+} \M(r, c_1, c_2)$,
whose slope is smaller than $\mu$ if $\mu < \mu(E)$, respectively, larger than $\mu$ if $\mu > \mu(E)$.
Iterating this process we obtain an infinite sequence of isomorphic moduli spaces.
We say that $\M(r, c_1, c_2)$ is \emph{initial} if it is not of the form $\Lambda^{+} \M(r, c_1', c_2')$
for any integers $c_1'$, $c_2'$.
The classification of the moduli spaces of height zero and dimension up to $10$ was
carried out in \cite{drezet_reine}.
We will see in Section \ref{3.1} that there are no moduli spaces of height zero and dimension $11$.

The first goal of this paper is to determine the additive structure of the homology groups
of the moduli spaces of height zero and dimension $12$.
In Section \ref{3.1} we will show that they are isomorphic to $\N(3, 4, 3)$, which is a smooth projective variety.
We will use the method of Bia{\l}ynicki-Birula \cite{birula}, \cite{birula_polonici},
which consists of analysing the fixed-point locus for the action of a torus on a smooth projective variety.
We refer to \cite[Section 2]{choi_maican} for a short introduction to the Bia{\l}ynicki-Birula theory.
We fix a vector space $V$ over $\CC$ of dimension $3$ and we identify $\PP^2$ with $\PP(V)$.
We fix a basis $\{ X, Y, Z \}$ of $V^*$.
We consider the action of $(\CC^*)^3$ on $\PP^2$ given by
\[
t (x_0^{} : x_1^{} : x_2^{}) =  (t_0^{-1} x_0^{} : t_1^{-1} x_1^{} : t_2^{-1} x_2^{}),
\]
where $t = (t_0, t_1, t_2)$. The induced action on the symmetric algebra of $V^*$ is given by the formula
\[
t  \, X^iY^jZ^k = t_0^i t_1^j t_2^k X^iY^jZ^k.
\]
In particular, we get an action of $(\CC^*)^3$ on $\Hom(\CC^4, \CC^3 \tensor V^*)$ by multiplication on $V^*$.
This descends to an action on $\N(3, 4, 3)$.
Consider the torus
\[
T = (\CC^*)^3/ \{ (c, c, c) \mid \ c \in \CC^* \}.
\]
Note that the action of $(\CC^*)^3$ on $\PP^2$ and, also, on $\N(3, 4, 3)$ factors through an action of $T$.

\begin{theorem}
\label{theorem_1}
The initial moduli spaces of height zero and dimension twelve are
$\M(5, -1, 4)$, $\M(7, -2, 6)$, $\M(5, 1, 4)$, $\M(7, 2, 6)$, as well as the moduli spaces obtained from these
by twisting with $\O(m)$, $m \in \ZZ$.
These moduli spaces are isomorphic to $\N(3, 4, 3)$. The $T$-fixed locus of $\N(3, 4, 3)$ consists of
$62$ isolated points and $3$ projective lines. Moreover, the integral homology groups of $\N(3, 4, 3)$ have no torsion
and its Poincar\'e polynomial is
\[
P(x) = x^{24} +  x^{22} + 3 x^{20} + 5 x^{18} + 8 x^{16} + 10 x^{14} + 12 x^{12} + 10 x^{10} + 8 x^8 + 5 x^6 + 3 x^4 + x^2 + 1.
\]
The Euler characteristic of $\N(3, 4, 3)$ is $68$ and its Hodge numbers satisfy the relation
$h^{pq} = 0$ if $p \neq q$.
\end{theorem}

\noindent
Dr\'ezet \cite{drezet_cohomologie} and Ellingsrud and Str{\o}mme \cite{ellingsrud-stromme}
have given algorithms for computing the homology of $\N(q, m, n)$ based on different methods.
Moreover, according to \cite{drezet_cohomologie}, the cohomology ring of $\N(q, m, n)$
is generated by the Chern clases of certain universal vector bundles, if $m$ and $n$ are coprime.
This implies the above statement about the Hodge numbers of $\N(3, 4, 3)$

Let $r > 0$ and $\chi$ be integers. Let $\M_{\PP^2}(r, \chi)$ be the moduli space of Gieseker
semi-stable sheaves on $\PP^2$ having Hilbert polynomial $P(m) = rm + \chi$
(this the same as $\M(0, r, r(r+3)/2 - \chi)$).
The aim of the second part of this paper is to compute the Hodge numbers of $\M_{\PP^2}(5,1)$.
According to \cite{lepotier_revue}, this is a smooth projective variety of dimension $26$.

The study of the moduli spaces $\M_{\PP^2}(r, 1)$ is partly motivated by Gromov-Witten Theory.
Let $X$ be a polarised Calabi-Yau threefold and fix $\beta \in \H_2(X, \ZZ)$.
Let $\M_X(\beta)$ be the moduli space of semi-stable sheaves $\F$ on $X$ having Euler characteristic $1$
and whose support has dimension $1$ and class $\beta$.
Consider the genus-zero Gromov-Witten invariant $N_{\beta}(X)$ of $X$
and the Donaldson-Thomas invariant $n_{\beta}(X) = \deg[\M_X(\beta)]^\vir$.
Katz \cite{katz_curves} conjectured the relation
\[
N_{\beta}(X) = \sum_{k | \beta} n_{\beta/k}(X) k^{-3},
\]
for which he found evidence by looking at contractible curves.
Li and Wu \cite{li_wu} have proved the conjecture in some particular cases.
Assume now that $X$ is the local $\PP^2$, that is, the total space of $\omega_{\PP^2}$.
Then, as noted in \cite{choi_thesis}, $\M_X(r) = \M_{\PP^2}(r, 1)$, hence, up to sign,
$n_r(X)$ is equal to the Euler characteristic of $\M_{\PP^2}(r, 1)$.
We refer to \cite{choi_maican} for a brief survey on the present state of research into these moduli spaces.

The Poincar\'e polynomial of $\M_{\PP^2}(5, 1)$ has already been computed in \cite{yuan}
by means of a cellular decomposition and in \cite{choi_chung} by the wall-crossing method.
We will, instead, apply the Bia{\l}ynicki-Birula method to the action on $\M_{\PP^2}(5, 1)$
induced by the action of $T$ on $\PP^2$.
Concretely, let $\mu_t \colon \PP^2 \to \PP^2$ denote the map of multiplication by $t \in T$
and let $[\F]$ denote the point in $\M_{\PP^2}(5,1)$ determined by a sheaf $\F$ on $\PP^2$.
The action of $T$ on the moduli space is given by
\[
t \, [\F] = [\mu_{t^{-1}}^* \F].
\]
To determine the torus fixed locus we will use the classification of semi-stable sheaves $\F$ on $\PP^2$
with Hilbert polynomial $P(m) = 5m+1$ provided at \cite{illinois}.
The same technique was used in \cite{homology} to study the homology of $\M_{\PP^2}(5, 3)$
and in \cite{choi_maican} to study $\M_{\PP^2}(4, 1)$.

\begin{theorem}
\label{theorem_2}
The $T$-fixed locus of $\M_{\PP^2}(5, 1)$ consists of $1407$ isolated points, $132$ projective lines
and $6$ irreducible components of dimension $2$ that are isomorphic to $\PP^1 \times \PP^1$.
The integral homology groups of $\M_{\PP^2}(5, 1)$ have no torsion and its Poincar\'e polynomial is
\begin{align*}
P(x) = & x^{52} + 2x^{50}+6x^{48}+13x^{46}+26x^{44}+45x^{42}+68x^{40}+87x^{38}+100x^{36} \\
& +107x^{34}+111x^{32}+112x^{30}+113x^{28}+113x^{26}+113x^{24}+112x^{22}+111x^{20} \\
& +107x^{18}+100x^{16}+87x^{14}+68x^{12}+45x^{10}+26x^8+13x^6+6x^4+2x^2+1.
\end{align*}
The Euler characteristic of $\M_{\PP^2}(5,1)$ is $1695$ and its Hodge numbers satisfy the relation
$h^{pq} = 0$ if $p \neq q$.
\end{theorem}

\noindent
The proof of Theorem \ref{theorem_2} relies to a large extent on the proof of Theorem \ref{theorem_1}.
Indeed, the moduli spaces $\M_{\PP^2}(5, 1)$ and $\N(3, 4, 3)$ are closely related.
According to \cite[Proposition 3.2.1]{illinois}, there is an open $T$-invariant subset $\M_0 \subset \M_{\PP^2}(5,1)$
that is isomorphic to an open subset inside a certain fibre bundle with base $\N(3, 4, 3)$ and fibre $\PP^{14}$.
The projection to the base $\M_0 \to \N(3,4,3)$ is $T$-equivariant,
surjective, and has surjective differential at every point.
Determining $\N(3, 4, 3)^T$ is, therefore, indispensable for determining $\M_0^T$. 
Moreover, as $T$-modules, the tangent space at a fixed point of $\N(3, 4, 3)$ is a direct summand
of the tangent space at a fixed point of $\M_0$ lying over it.

The paper is organised as follows. In Sections \ref{section_2} and \ref{section_4}
we will determine the torus fixed points in $\N(3, 4, 3)$, respectively, $\M_{\PP^2}(5, 1)$.
In Sections \ref{section_3} and \ref{section_5} we will describe the torus action on the tangent
spaces at the fixed points. We will omit the proofs of most propositions from Sections \ref{section_4}
and \ref{section_5} because they are analogous to results in \cite{choi_maican} and \cite{homology}.

In closing, we mention that the Betti numbers of $\M_{\PP^2}(5, 1)$ are the same as the Betti numbers of
$\M_{\PP^2}(5, 3)$, as computed in \cite{yuan} and \cite{homology}.
This raises the question whether these two moduli spaces are isomorphic.

\section*{Acknowledgements}

\noindent
The author was supported by
Consiliul Na\c{t}ional al Cercet\u{a}rii \c{S}tiin\c{t}ifice (Romania), grant PN II--RU 169/2010 PD--219.
The author has benefitted from discussions with Jinwon Choi.
The referee pointed out several omissions in the Introduction, numerous improvements
to the presentation, especially in Section \ref{section_2}, and a significant error in Section \ref{3.1}.


\section{The torus fixed locus of $\N(3, 4, 3)$}
\label{section_2}

Consider the vector space
\[
W = \Hom(4\O_{\PP^2}(-2), 3 \O_{\PP^2}(-1)) = \Hom(\CC^4 \tensor V, \ \CC^3).
\]
Its elements are represented by $3 \times 4$-matrices $\f$ with entries in $V^*$.
The reductive group
\[
G = (\GL(4,\CC) \times \GL(3,\CC))/\CC^*
\]
acts on $W$ by the formula $(g, h) \f = h \f g^{-1}$. Here $\CC^*$ is embedded as the
subgroup of homotheties. According to King's Criterion of Semi-stability \cite{king},
the set $W^\ss$ of semi-stable elements consists of those matrices that are not in the orbit
of a matrix having a zero-column, or a zero $1 \times 3$-submatrix, or a zero $2 \times 2$-submatrix.
We have a geometric quotient $W^\ss/G$, which is the Kronecker moduli space $\N(3, 4, 3)$.
The point $[\f]$ in the moduli space determined by $\f$ is $T$-fixed if and only if for each $t \in T$
there is $(g, h) \in G$ such that $t \f = (g, h) \f$. Let $\zeta_i$, $1 \le i \le 4$, denote the maximal minor of $\f$
obtained by deleting the column $i$.
Let $W_i \subset W^\ss$, $i = 0, 1, 2$, denote the subset given by the condition
\[
\deg(\gcd(\zeta_1, \zeta_2, \zeta_3, \zeta_4)) = i.
\]
Its image $\N_i$ in $\N(3, 4, 3)$ is torus-invariant.
Moreover, $\N_0$ is an open subset, $\N_1$ is a locally closed subset and $\N_2$ is a closed subset of $\N(3, 4, 3)$.
There are no semi-stable matrices $\f$ whose maximal minors have a common factor of degree $3$.
Indeed, if $\f$ were such a matrix, then we might assume, after we perform elementary column operations,
that $\zeta_1 = 0$, $\zeta_2 = 0$, $\zeta_3 = 0$, $\zeta_4 \neq 0$. But
\[
\f \, \left[
\ba{cccc}
\zeta_1 & - \zeta_2 & \zeta_3 & - \zeta_4
\ea
\right]^\TR = 0,
\]
showing that the last column of $\f$ is zero, which is contrary to semi-stability.
Also, the case when $\zeta_i$ are all zero is not feasible.
Indeed, assume that all maximal minors of $\f$ were zero.
Let $\f_i$ denote the matrix obtained from $\f$ by deleting column $i$.
By hypothesis, $\f_i$ is not equivalent to a matrix having a zero row, a zero column,
or a zero $2 \times 2$-submatrix. In other words, $\f_i$ is semi-stable as a Kronecker module.
From the description of $\N(3, 3, 3)$ found in \cite{lepotier_revue}, we know that
\[
\f_i \sim \left[
\ba{ccc}
Y & X & 0 \\
Z & 0 & X \\
0 & Z & -Y
\ea
\right].
\]
Thus, there is a row vector $l_i$ of length $3$ whose entries form a basis of $V^*$, such that
$l_i \f_i = 0$. Since $\f_i$ and $\f_j$ have two columns in common, it is easy to see that
$l_i = l_j$ for all $1 \le i < j \le 4$. It follows that $l_1 \f = 0$, hence $\f$ has linearly dependent columns,
contrary to semi-stability.
The above discussion shows that
\[
\N_0 \cup \N_1 \cup \N_2 = \N(3, 4, 3).
\]

\subsection*{Notations}

\noindent \\ \\
\begin{tabular}{r c l}
$\N \ $ & = & $\N(3, 4, 3)$; \\
$\Sym_3$ & = & the group of permutations $\sigma$ of the variables $(X,Y,Z)$; \\
$\A_3$ & = & the subgroup of $\Sym_3$ of even permutations; \\
$\f_{\sigma}$ & = & the matrix obtained from $\f$ after performing the permutation $\sigma$; \\
$U_{\f}$ & = & $\spann \{ \zeta_1, \zeta_2, \zeta_3, \zeta_4 \} \subset \Sym^3 V^*$; \\
$\type(\f)$ & = & $(\rank(\f \mod (Y,Z)), \rank(\f \mod (X,Z)), \rank(\f \mod (X,Y)))$; \\
$p_0$ & = & $(1 : 0 : 0)$; \\
$p_1$ & = & $(0 : 1 : 0)$; \\
$p_2$ & = & $(0 : 0 : 1)$; \\
$p_{ij}$ & = & the double point supported on $p_i$ and contained in the line $p_i p_j$; \\
$q_i$ & = & the triple point supported on $p_i$, that is not contained in a line; \\
$q_{ij}$ & = & the triple point supported on $p_i$ and contained in the line $p_i p_j$.
\end{tabular}

\subsection{Fixed points in $\N_0$}    
\label{2.1}

Consider the Hilbert scheme of zero-dimensional subschemes of $\PP^2$ of length $6$.
Let $\Hilb^0_{\PP^2}(6)$ be the open subset of schemes that are not contained in a conic curve.
According to \cite[Propositions 4.5 and 4.6]{modules-alternatives}, there is a $T$-equivariant isomorphism
\[
\N_0 \isom \Hilb^0_{\PP^2}(6), \qquad [\f] \mapsto \Z,
\]
where the ideal sheaf of $\Z$ is the cokernel of $\f^\TR \colon 3\O(-4) \to 4\O(-3)$.
Equivalently, $\Z$ is the zero-set of the ideal generated by the maximal minors $\zeta_1$, $\zeta_2$, $\zeta_3$, $\zeta_4$ of $\f$.
Thus, up to equivalence, $\f$ is uniquely determined by its maximal minors.
The classification of the $T$-invariant subschemes $\Z \in \Hilb^0_{\PP^2}(6)$ is well-known.
Note that $\Z$ is supported on a subset of $\{ p_0, p_1, p_2 \}$, so it has one of the following types:
$(2,2,2)$, $(1,2,3)$, $(1,1,4)$, $(1,5)$, $(2,4)$, $(3,3)$, $6$.

\subsubsection{Type $(2,2,2)$} 
\label{2.1.1}

There are two $T$-invariant subschemes $\Z$ of type $(2,2,2)$, namely $\{ p_{01}, p_{12}, p_{20} \}$ and
$\{ p_{02}, p_{21}, p_{10} \}$.
The ideal of the former is
\[
(Y^2, Z) \cap (Z^2, X) \cap (X^2, Y) = (XYZ, YZ^2, X^2 Z, XY^2),
\]
so it corresponds to the matrix
\[
\boldsymbol{\alpha} =
{\mathversion{bold}
\left[
\ba{cccc}
Z & X & 0 & 0 \\
X & 0 & Y & 0 \\
Y & 0 & 0 & Z
\ea
\right]}
\]
in $\N_0$. The matrix corresponding to the other scheme is $\alpha_{\sigma}$, where $\sigma$ is the transposition $(X,Y)$.

\subsubsection{Type $(1,2,3)$} 
\label{2.1.2}

Note that $\Z$ cannot contain a subscheme of the form $q_{ij}$, otherwise $\Z$ would be contained
in the union of two lines. Thus, $\Z$ contains a subscheme of the form $q_i$, say $q_2$.
Note that $\{ p_0, p_{12}, q_2 \}$ is contained in $p_0 p_2 \cup p_1 p_2$.
Thus, there are six schemes obtained from $\Z = \{ p_0, p_{10}, q_2 \}$ by the action of $\Sym_3$.
Since
\[
I(\Z) = (Y,Z) \cap (X^2, Z) \cap (X^2, XY, Y^2) = (XYZ, Y^2 Z, X^2 Z, X^2 Y),
\]
we see that $\Z$ corresponds to the matrix
\[
\boldsymbol{\beta} =
{\mathversion{bold}
\left[
\ba{cccc}
Y & X & 0 & 0 \\
X & 0 & Y & 0 \\
X & 0 & 0 & Z
\ea
\right]}.
\]

\subsubsection{Type $(3,3)$} 
\label{2.1.3}

The two subschemes of $\Z$ of length $3$ cannot be each contained in a line.
The scheme $\{ q_0, q_1 \}$ is unfeasible because it is contained in the conic curve $\{ Z^2 = 0\}$.
It follows that $\Z = \{ q_i, q_{jk} \}$, for distinct $i, j, k$.
There are six such schemes obtained from $\Z = \{ q_{01}, q_2 \}$ by the action of $\Sym_3$.
Since
\[
I(\Z) = (Y^3, Z) \cap (X^2, XY, Y^2) = (Y^2 Z, XYZ, X^2 Z, Y^3),
\]
we deduce that $\Z$ corresponds to the matrix
\[
\boldsymbol{\gamma} =
{\mathversion{bold}
\left[
\ba{cccc}
Y & X & 0 & 0 \\
X & 0 & Y & 0 \\
0 & Y & 0 & Z
\ea
\right]}.
\]

\subsubsection{Type $(1,1,4)$}  
\label{2.1.4}

Let $r$ be the point of $\Z$ of multiplicity $4$, supported say at $p_2$.
Note that $r$ cannot be contained in a line, otherwise $\Z$ would be contained in the union
of two lines.
Thus, $I(r) \cap \spann\{ X^2, XY, Y^2 \}$ is a $T$-invariant subspace of dimension $2$.
This subspace must be generated by two invariant monomials.
Note that $r$ cannot be contained in the conic curve $\{ XY = 0 \}$, otherwise $\Z$ would be contained
in the said conic.
Thus, $I(r) = (X^2, Y^2)$; we denote this point by $r_2$.
We obtain three schemes: $\{ r_0, p_1, p_2 \}$, $\{ r_1, p_0, p_2 \}$, $\{ r_2, p_0, p_1 \}$.
The ideal of the latter is
\[
(X, Z) \cap (Y, Z) \cap (X^2, Y^2) = (Y^2 Z, X^2 Z, X Y^2, X^2 Y),
\]
so it corresponds to the matrix
\[
\boldsymbol{\delta} =
{\mathversion{bold}
\left[
\ba{cccc}
Y & Z & 0 & 0 \\
0 & Y & X & 0 \\
0 & 0 & Z & X
\ea
\right]}.
\]
The other two matrices are $\delta_{\sigma}$, where $\sigma = (X,Z)$, respectively, $(Y,Z)$.

\subsubsection{Type $(2,4)$}   
\label{2.1.5}

Let $r$ be the point of $\Z$ of multiplicity $4$.
As before, $r$ cannot be contained in a line.
Moreover, $r \neq r_2$, otherwise $\Z$ would be contained in the conic curve
$\{ X^2 = 0 \}$ or $\{ Y^2 = 0 \}$.
Thus, $r= r_{20}$ or $r=r_{21}$, where $r_{20}$ is given by the ideal $(X^3, XY, Y^2)$
and $r_{21}$ is given by the ideal $(X^2, XY, Y^3)$.
We get six schemes obtained from $\Z = \{ p_{10}, r_{20} \}$ by the action of $\Sym_3$.
Since
\[
I(\Z) = (X^2, Z) \cap (X^3, XY, Y^2) = (Y^2 Z, XYZ, X^2 Y, X^3),
\]
we deduce that $\Z$ corresponds to the matrix
\[
\boldsymbol{\varepsilon} =
{\mathversion{bold}
\left[
\ba{cccc}
X & Y & 0 & 0 \\
Z & 0 & X & 0 \\
0 & 0 & Y & X
\ea
\right]}.
\]

\subsubsection{Type $(1,5)$}     
\label{2.1.6}

Let $s$ be the point of $\Z$ of multiplicity $5$, supported say on $p_2$.
Note that $s$ cannot be contained in a line or in the conic $\{ XY = 0\}$.
It follows that $s = s_{20}$ or $s = s_{21}$, where $s_{20}$ is given by the ideal
$(X^3, X^2 Y, Y^2)$ and $s_{21}$ is given by the ideal $(X^2, X Y^2, Y^3)$.
Letting $\Sym_3$ act on $\Z= \{ p_1, s_{20} \}$ we obtain six $T$-fixed schemes.
The ideal of $\Z$ is
\[
(X,Z) \cap (X^3, X^2 Y, Y^2) = (Y^2 Z, X Y^2, X^2 Y, X^3),
\]
hence $\Z$ corresponds to the matrix
\[
\boldsymbol{\zeta} =
{\mathversion{bold}
\left[
\ba{cccc}
Y & X & 0 & 0 \\
X & 0 & Y & 0 \\
0 & Z & 0 & X
\ea
\right]}.
\]

\subsubsection{Type $6$}  
\label{2.1.7}

There is only one $T$-invariant point of multiplicity $6$ supported on $p_2$.
Its ideal is $(X^3, X^2 Y, X Y^2, Y^3)$ and the associated matrix is
\[
\boldsymbol{\eta} =
{\mathversion{bold}
\left[
\ba{cccc}
Y & X & 0 & 0 \\
X & 0 & Y & 0 \\
0 & 0 & X & Y
\ea
\right]}.
\]
The matrices corresponding to the other two points of multiplicity $6$
are $\eta_{\sigma}$, where $\sigma = (X, Z)$, respectively, $(Y,Z)$.


\subsection{Fixed points in $\N_1$} 
\label{2.2}

Consider $\f \in W$.
Assume that $\gcd(\zeta_1, \zeta_2, \zeta_3, \zeta_4) = l$ for some $l \in V^*$.
Let $L \subset \PP^2$ denote the line given by the equation $l = 0$.
According to \cite[Proposition 3.3.5]{illinois}, $\f$ is semi-stable if and only if $\zeta_1$, $\zeta_2$, $\zeta_3$, $\zeta_4$
are linearly independent, which we assume in the sequel.
Consider the vector space $U_{\f} = \spann \{ \zeta_1, \zeta_2, \zeta_3, \zeta_4 \} \subset \Sym^3 V^*$.

\begin{proposition}
\label{2.2.1}
Assume that $[\f]$ is $T$-invariant.
Then $l \in \{ X, Y, Z \}$ and any monomial occurring in any $\zeta_i$, $1 \le i \le 4$,
belongs to $U_{\f}$. Moreover, there are distinct monomials
$\xi_1, \xi_2, \xi_3, \xi_4 \in \Sym^3 V^*$ such that
$U_{\f} = \spann \{ \xi_1, \xi_2, \xi_3, \xi_4 \}$.
\end{proposition}

\begin{proof}
For every $t \in T$, $t \f \sim \f$, hence $U_{t \f} = U_{\f}$,
hence $\spann \{ t l \} = \spann \{ l \}$. Thus, $l$ is a monomial.
Let $\xi_1, \ldots, \xi_k$ be the distinct monomials occurring in $\zeta_1$.
We can find $t_1, \ldots, t_k \in T$ such that $t_1 \zeta_1, \ldots, t_k \zeta_1$
are linearly independent. Thus, each $\xi_i$, $1 \le i \le k$, is a linear combination
of $t_1 \zeta_1, \ldots, t_k \zeta_1$, so it belongs to $U_{\f}$.
This shows that exactly four monomials occur in $\zeta_1, \zeta_2, \zeta_3, \zeta_4$,
and these monomials span $U_{\f}$.
\end{proof}

\noindent
Arguing as in the proof of \cite[Proposition 3.3.6]{illinois}, we can show that,
after performing elementary row and column operations, we may write
\[
\f = \left[
\ba{cccc}
 & \upsilon & & \ba{c} 0 \\ 0 \ea \\
\star & \star & \star & l'
\ea
\right],
\]
where $\upsilon$ has linearly independent maximal minors.
Thus, $\upsilon$ gives a point in $\N(3, 3, 2)$.
We will examine several cases according to the different possibilities for $\Coker(\upsilon)$.

Assume first that the maximal minors of $\upsilon$ have no common factor.
This condition defines an open subset $\N_0(3, 3, 2) \subset \N(3, 3, 2)$.
Consider the Hilbert scheme of zero-dimensional subschemes of length $3$ of $\PP^2$
and let $\Hilb^0_{\PP^2}(3)$ denote the open subset of subschemes that are not contained in a line.
According to \cite[Propositions 4.5 and 4.6]{modules-alternatives},
we have an isomorphism
\[
\N_0(3, 3, 2) \isom \Hilb^0_{\PP^2}(3), \qquad [\upsilon] \mapsto \Y,
\]
where the ideal sheaf of $\Y$ is the cokernel of $\upsilon^\TR \colon 2\O(-3) \to 3\O(-2)$.
Equivalently, $\Y$ is the zero-set of the ideal generated by the maximal minors of $\upsilon$.
Thus, up to equivalence, $\upsilon$ is uniquely determined by its maximal minors.
Note that $l' = l$. Without loss of generality we may assume that $l' =X$, the other cases
being obtained by a permutation of the variables.
We have an extension
\[
0 \lra \O_L(-1) \lra \Coker(\f) \lra \O_{\Y} \lra 0.
\]
In the sequel $\f$ will represent a $T$-fixed point in $\N(3,4,3)$.

\setcounter{subsubsection}{1}

\subsubsection{The case when $\Y$ is the union of three non-colinear closed points} 
\label{2.2.2}

Let $\Y' \subset \Y$ be the subscheme supported on $\PP^2 \setminus L$.
For $t \in T$ denote by $\mu_t \colon \PP^2 \to \PP^2$ the  map of multiplication by $t$.
Note that $\mu_t^* (\Coker(\f)) \isom \Coker(\f)$, hence, taking into account the exact sequence from above,
$\mu_t^* (\O_{\Y'}) \isom \O_{\Y'}$.
We deduce that $\Y' = \{ p_0 \}$.
The other two points of $\Y$ are given by the ideals
$(X, aY + bZ)$, respectively, $(X, cY + dZ)$, where $ad - bc \neq 0$.
It follows that
\begin{align*}
I(\Y) & = (aY + bZ, cY + dZ) \cap (X, aY + bZ) \cap (X, cY + dZ) \\
 & = (X(aY + bZ), \, X(cY + dZ), \, (aY + bZ)(cY + dZ)).
\end{align*}
Thus, we may write
\[
\f = \left[
\ba{cccc}
X & aY+bZ & 0 & 0 \\
X & 0 & cY+dZ & 0 \\
l_1 & l_2 & l_3 & X
\ea
\right].
\]
Since $X$ divides $\zeta_4$, $X$ also divides $l_1$, so, performing column operations, we may assume that $l_1 = 0$.
Note that $a, b, c, d$ cannot be all non-zero, otherwise $U_{\f}$ would contain the set $X \{ XY, XZ, Y^2, YZ, Z^2 \}$,
which is contrary to Proposition \ref{2.2.1}.
Assume that $b=0$. Thus, $a \neq 0$, $d \neq 0$ and
\[
\f \sim \left[
\ba{cccc}
X & Y & 0 & 0 \\
X & 0 & cY+Z & 0 \\
0 & eZ & fY & X
\ea
\right].
\]
If $e \neq 0$, then $c = 0$, otherwise $X \{ Y^2, YZ, XY, XZ, Z^2 \} \subset U_{\f}$, contradicting Proposition \ref{2.2.1}.
Analogously, $f = 0$.
We obtain the matrix
\[
\boldsymbol{\theta} =
{\mathversion{bold}
\left[
\ba{cccc}
X & Y & 0 & 0 \\
X & 0 & Z & 0 \\
0 & Z & 0 & X
\ea
\right]},
\]
which clearly represents a $T$-fixed point.
If $e=0$, then, by semi-stability, $f \neq 0$, so
\[
\f \sim \left[
\ba{cccc}
X & Y & 0 & 0 \\
X & 0 & cY + Z & 0 \\
0 & 0 & Y & X
\ea
\right] \sim \left[
\ba{cccc}
X & Y & 0 & 0 \\
X & 0 & Z & 0 \\
0 & 0 & Y & X
\ea
\right] \sim \theta_{(Y, Z)}.
\]
In conclusion, we have six $T$-fixed points
{\mathversion{bold} $[\theta_{\sigma}]$, $\sigma \in \Sym_3$}.
They are distinct because $U_{\theta} = \spann \{ XYZ, X^2 Z, X^2 Y, XZ^2 \}$ is not fixed
by any $\sigma \in \Sym_3$.

\subsubsection{The case when $\Y$ is the union of a closed point on $L$ and a double point outside $L$} 
\label{2.2.3}

Assume that $\Y = \{ p, q \}$, where $p$ is a closed point and $q$ is a double point.
Arguing as at Section \ref{2.2.2}, we see that $p$ and $q$ cannot be both in $\PP^2 \setminus L$.
Assume that $p \in L$ and $q \in \PP^2 \setminus L$.
Then $q \in \{ p_{01}, p_{02} \}$, say $q = p_{02}$.
Thus, $p$ is given by the ideal $(X, aY+Z)$ for some $a \in \CC$
and $q$ is given by the ideal $(Y, Z(aY+Z))$.
We have
\[
I(\Y) = (X, aY + Z) \cap (Y, Z(aY + Z)) = (XY, \, Y(aY + Z), \, Z(aY + Z)),
\]
so we may write
\[
\f = \left[
\ba{cccc}
Y & Z & 0 & 0 \\
0 & X & aY+Z & 0 \\
l_1 & l_2 & l_3 & X
\ea
\right].
\]
Denote $l_i = a_i X + b_i Y + c_i Z$, $i = 1, 2, 3$.
Since $X$ divides $\zeta_4$, it follows that $X$ divides $Y l_2 - Z l_1$,
hence $X$ divides $Y (b_2 Y + c_2 Z) - Z (b_1 Y + c_1 Z)$,
hence $b_2 = 0$, $c_1 = 0$, $b_1 = c_2$.
Performing elementary row and column operations, we may write
\[
\f = \left[
\ba{cccc}
Y & Z & 0 & 0 \\
0 & X & aY + Z & 0 \\
0 & 0 & l_3 & X
\ea
\right].
\]
Performing elementary operations on $\f$
we may assume that $a_3 = 0$ and $c_3 = 0$. Note that $b_3 \neq 0$, otherwise the semi-stability of $\f$
would get contradicted. Thus, we may write
\[
\f = \left[
\ba{cccc}
Y & Z & 0 & 0 \\
0 & X & aY+Z & 0 \\
0 & 0 & Y & X
\ea
\right]. \quad \text{Consider} \quad \psi_c = \left[
\ba{cccc}
Y & Z & 0 & 0 \\
0 & X & acY+Z & 0 \\
0 & 0 & Y & X
\ea
\right]
\]
for each $c \in \CC^*$.
Note that $\{ [t \f], t \in T \} = \{ [\psi_c], c \in \CC^* \}$, hence $\f \sim \psi_c$ for all $c \in \CC^*$.
The orbits for the action of $G$ on $W^\ss$ are closed because there are no properly
semi-stable points. Thus, ${\displaystyle \f \sim \lim_{c \to 0} \psi_c = \iota}$, where
\[
\boldsymbol{\iota} =
{\mathversion{bold}
\left[
\ba{cccc}
Y & Z & 0 & 0 \\
0 & X & Z & 0 \\
0 & 0 & Y & X
\ea
\right]}
\]
clearly determines a $T$-fixed point in $\N$.
Note that $\iota \nsim \theta_{\sigma}$ for all $\sigma \in \Sym_3$ because $\type(\theta) = (2,1,2)$ whereas
$\type(\iota) = (2,2,2)$.
Thus, we obtain six new fixed points {\mathversion{bold} $[\iota_{\sigma}]$, $\sigma \in \Sym_3$}.
They are distinct because $U_{\iota} = \spann \{ XZ^2, XYZ, X^2 Y, X Y^2 \}$
is not fixed by any $\sigma \in \Sym_3$.

\subsubsection{The case when $\Y$ is the union of a closed point outside $L$ and a double point} 
\label{2.2.4}

Assume that $\Y = \{ p, q \}$, where $p \in \PP^2 \setminus L$ is a closed point and $q$ is a double point.
Then, as before, $p=p_0$ and $\red(q) \in L$. In fact, we will show that $q$ is a subscheme of $L$.
Assume that the contrary is true.
Then $I(q)= ((aY+bZ)^2, X+aY+bZ)$, where $a$ and $b$ are not both zero, say $a \neq 0$.
We have $I(p) = (aY + bZ, Z)$, hence
\[
I(\Y) = (Z(X + aY + bZ), \, (aY + bZ)(X + aY + bZ), \, (aY + bZ)^2).
\]
We may now write
\[
\f = \left[
\ba{cccc}
X + aY + bZ & aY+bZ & 0 & 0 \\
0 & Z & aY+bZ & 0 \\
\star & \star & \star & X
\ea
\right].
\]
Note that $b=0$, otherwise $X \{ Y^2, YZ, Z^2, XY, XZ \} \subset U_{\f}$, contrary to Proposition \ref{2.2.1}.
After we perform elementary row and column operations, we may write
\[
\f = \left[
\ba{cccc}
X + aY & Y & 0 & 0 \\
0 & Z & Y & 0 \\
cY+dZ & 0 & eY+fZ & X
\ea
\right].
\]
Note that $U_{\f} = \spann \{ XY^2, X^2Y, X^2Z, XYZ \}$.
Since
\[
\zeta_4= eXYZ + fXZ^2 + (ae+d) Y^2 Z + afYZ^2 + cY^3
\]
we deduce, by virtue of Proposition \ref{2.2.1}, that $f=0$, $ae+d=0$, $c=0$. Thus, $e \neq 0$ and
\[
\f \sim 
\left[
\ba{cccc}
X + aY & Y & 0 & 0 \\
0 & Z & Y & 0 \\
-aZ & 0 & Y & X
\ea
\right] \sim \left[
\ba{cccc}
X & Y & 0 & 0 \\
-aZ & Z & Y & 0 \\
0 & -Z & 0 & X
\ea
\right].
\]
For each $c \in \CC^*$ consider the morphism
\[
\psi_c = \left[
\ba{cccc}
X & Y & 0 & 0 \\
-acZ & Z & Y & 0 \\
0 & -Z & 0 & X
\ea
\right].
\]
Note that ${\displaystyle \psi = \lim_{c \to 0} \psi_c}$ belongs to $W^\ss$
because $U_{\psi} = \spann \{ XY^2, X^2Y, X^2Z, XYZ \}$ has dimension $4$.
By the argument at Section \ref{2.2.3}, we deduce that $\f \sim \psi$.
However, this is absurd, because $\type(\f) \neq \type(\psi)$.

The above discussion shows that $q$ is a subscheme of $L$, so $I(q) = (X, (aY+bZ)^2)$,
where $a$, $b$ are not both zero, say $a \neq 0$. We have
\[
I(\Y) = (aY + bZ, Z) \cap I(q) = (XZ, \, X(aY + bZ), \, (aY + bZ)^2),
\]
so we may write
\[
\f = \left[
\ba{cccc}
X & aY+bZ & 0 & 0 \\
0 & Z & aY+bZ & 0 \\
\star & \star & \star & X
\ea
\right].
\]
Note that $X \{ aY^2, abYZ, bZ^2, aXY, XZ \} \subset U_{\f}$, hence, by Proposition \ref{2.2.1},
$b=0$, that is, $q=p_{21}$.
We may write
\[
\f = \left[
\ba{cccc}
X & Y & 0 & 0 \\
0 & Z & Y & 0 \\
l_1 & l_2 & l_3 & X
\ea
\right].
\]
Since $X$ divides $\zeta_4$, $X$ also divides $l_1$, hence, performing row and column operations, we may write
\[
\f = \left[
\ba{cccc}
X & Y & 0 & 0 \\
0 & Z & Y & 0 \\
0 & 0 & cY+dZ & X
\ea
\right].
\]
Note that $X \{ Y^2, XY, XZ, cYZ, dZ^2 \} \subset U_{\f}$, hence, in view of Proposition \ref{2.2.1},
$c=0$ or $d=0$. If $d=0$, then
\[
\f \sim \left[
\ba{cccc}
X & Y & 0 & 0 \\
0 & Z & Y & 0 \\
0 & 0 & Y & X
\ea
\right] \sim \left[
\ba{cccc}
X & Y & 0 & 0 \\
0 & Z & 0 & X \\
0 & 0 & Y & X
\ea
\right] \sim \left[
\ba{cccc}
X & 0 & Y & 0 \\
0 & Y & 0 & X \\
0 & 0 & Z & X
\ea
\right] \sim \theta_{(Y, Z)}.
\]
If $c=0$, then we obtain the matrix
\[
\boldsymbol{\kappa} =
{\mathversion{bold}
\left[
\ba{cccc}
X & Y & 0 & 0 \\
0 & Z & Y & 0 \\
0 & 0 & Z & X
\ea
\right]}
\]
representing a $T$-fixed point.
Note that $U_{\kappa} = \spann \{ XY^2, X^2Y, X^2Z, XZ^2 \} \neq U_{\iota_{\sigma}}$,
hence $[\kappa] \neq [\iota_{\sigma}]$ for all $\sigma \in \Sym_3$.
Moreover considering types we see that $[\kappa] \neq [\theta_{\sigma}]$ for all $\sigma \in \Sym_3$.
Since $\kappa \sim \kappa_{(Y,Z)}$, we conclude that we get three new $T$-fixed points,
{\mathversion{bold} $[\kappa_{\sigma}]$, $\sigma \in \A_3$}.

\subsubsection{The case when $\Y$ is the union of a closed point on $L$ and a double point whose
support is on $L$}    
\label{2.2.5}

Assume that $\Y = \{ p, q \}$, where $p, \red(q) \in L$.
Thus,
\[
I(q) = (X^2, eX+aY+bZ), \qquad I(p) = (X, cY+dZ),
\]
where $ad - bc \neq 0$. It follows that
\[
I(\Y) = (X(eX + aY + bZ), \, (cY + dZ)(eX + aY + bZ), \, X^2),
\]
so we may write
\[
\f = \left[
\ba{cccc}
eX+aY+bZ & X & 0 & 0 \\
0 & cY+dZ & X & 0 \\
\star & \star & \star & X
\ea
\right].
\]
Note that $X \{ X^2, aXY, bXZ, acY^2, bdZ^2, (ad+bc)YZ \} \subset U_{\f}$,
hence, by Proposition \ref{2.2.1}, $a=0$ or $b=0$.
We may assume that $b=0$, the case when $a=0$ being obtained by a permutation of variables.
Thus, $d \neq 0$.
Assume first that $e \neq 0$.
Then $X \{ X^2, XY, XZ, YZ, acY^2 \} \subset U_{\f}$, forcing $c=0$, in view of Proposition \ref{2.2.1}.
We may write
\[
\f = \left[
\ba{cccc}
eX + Y & X & 0 & 0 \\
0 & Z & X & 0 \\
l_1 & l_2 & l_3 & X
\ea
\right].
\]
Since $X$ divides $\zeta_4$, we see that $X$ divides $l_3$, so, performing row and column operations,
we may assume that $l_3 =0$ and $l_2 = fY$.
Since $fXY^2$ belongs to $U_{\f} = \spann \{ X^3, X^2Y, X^2Z, XYZ \}$, we see that $f=0$, hence
\[
\f \sim \left[
\ba{cccc}
eX+Y & X & 0 & 0 \\
0 & Z & X & 0 \\
Z & 0 & 0 & X
\ea
\right] \sim \left[
\ba{cccc}
Y & X & 0 & 0 \\
-eZ & Z & X & 0 \\
Z & 0 & 0 & X
\ea
\right] \sim
{\mathversion{bold}
\left[
\ba{cccc}
Y & X & 0 & 0 \\
0 & Z & X & 0 \\
Z & 0 & 0 & X
\ea
\right]} = \boldsymbol{\lambda}.
\]
Clearly, $\lambda$ represents a $T$-fixed point of $\N$ different from
$\theta_{\sigma}$, $\iota_{\sigma}$, $\kappa_{\sigma}$ for all $\sigma \in \Sym_3$
because $\type(\lambda)=(3,1,2)$ is different from the types of the other points.
We obtain six new $T$-fixed points {\mathversion{bold} $[\lambda_{\sigma}]$, $\sigma \in \Sym_3$}.

Assume now that $e=0$. As $b = 0$, we have $a \neq 0$.
Thus, we may write
\[
\f = \left[
\ba{cccc}
Y & X & 0 & 0 \\
0 & cY+Z & X & 0 \\
l_1 & l_2 & l_3 & X
\ea
\right].
\]
As before, $X$ divides $l_3$, hence, performing row and column operations, we may write
\[
\f = \left[
\ba{cccc}
Y & X & 0 & 0 \\
0 & cY+Z & X & 0 \\
gZ & fY & 0 & X
\ea
\right].
\]
Assume first that $c \neq 0$.
Note that $X \{ X^2, XY, Y^2, YZ, gXZ \} \subset U_{\f}$, hence, in view of Proposition \ref{2.2.1}, $g = 0$ and
\[
\f \sim 
\left[
\ba{cccc}
Y & X & 0 & 0 \\
0 & cY+Z & X & 0 \\
0 & Y & 0 & X
\ea
\right] \sim
{\mathversion{bold}
\left[
\ba{cccc}
Y & X & 0 & 0 \\
0 & Z & X & 0 \\
0 & Y & 0 & X
\ea
\right] = \boldsymbol{\mu}}.
\]
Clearly, $[\mu]$ is $T$-fixed. Considering types we can see that
$[\mu] \neq [\theta_{\sigma}]$, $[\iota_{\sigma}]$, $[\kappa_{\sigma}]$ for all $\sigma \in \Sym_3$.
Moreover, $[\mu] \neq [\lambda_{(Y,Z)}]$ because $U_{\mu} \neq U_{\lambda_{(Y,Z)}}$.
Thus, we obtain six new points {\mathversion{bold} $[\mu_{\sigma}]$, $\sigma \in \Sym_3$}.
Finally, we assume that $c=0$, so we may write
\[
\f = 
\left[
\ba{cccc}
Y & X & 0 & 0 \\
0 & Z & X & 0 \\
gZ & fY & 0 & X
\ea
\right].
\]
We have $U_{\f} = \spann \{ X^3, X^2Y, XYZ, X(fY^2 - gXZ) \}$, hence, by Proposition \ref{2.2.1}, $f=0$ or $g=0$.
We obtain the fixed points $[\lambda]$, respectively, $[\mu]$.

\subsubsection{The case when $\Y$ is a triple point}  
\label{2.2.6}

Assume that $\Y$ is a triple point that is not contained in a line.
If $\Y \subset \PP^2 \setminus L$, then $\Y$ is $T$-fixed, hence $\Y = \{ q_0 \}$ and
$I(\Y) = (Y^2, YZ, Z^2)$.
Recalling from the beginning of Section \ref{2.2} that the generators of $I(\Y)$ determine $\upsilon$ up to equivalence, we may write
\[
\f = \left[
\ba{cccc}
Y & Z & 0 & 0 \\
0 & Y & Z & 0 \\
l_1 & l_2 & l_3 & X
\ea
\right],
\]
where $l_1, l_2, l_3 \in \CC[Y,Z]$.
Since $X$ divides $\zeta_4$ and $\zeta_4 \in \CC[Y,Z]$, we deduce that $\zeta_4=0$,
which is contrary to our assumption that $\f$ give a point in $\N_1$.
This shows that $\red(\Y)$ is a point on $L$,
given by the ideal $(X, aY+bZ)$, where $a$, $b$ are not both zero, say $a \neq 0$.
The possible ideals defining $\Y$ are
\begin{enumerate}
\item[(i)] $(X^2, X(aY+bZ), (aY+bZ)^2)$,
\item[(ii)] $(X^2, X(aY+bZ), (aY+bZ)^2 - XZ)$,
\item[(iii)] $((cX+aY+bZ)^2, (cX+aY+bZ)X, X^2 - (cX+aY+bZ)Z)$.
\end{enumerate}
In the first case we may write
\[
\f = \left[
\ba{cccc}
X & aY+bZ & 0 & 0 \\
0 & X & aY+bZ & 0 \\
\star & \star & \star & X
\ea
\right].
\]
Note that $X \{ a^2 Y^2, ab YZ, b^2 Z^2, aXY, b XZ, X^2 \} \subset U_{\f}$,
hence, by Proposition \ref{2.2.1}, $b=0$, that is $\Y = \{ q_2 \}$. We may now write
\[
\f = \left[
\ba{cccc}
X & Y & 0 & 0 \\
0 & X & Y & 0 \\
l_1 & l_2 & l_3 & X
\ea
\right].
\]
By hypothesis $X$ divides $\zeta_4$, hence $X$ divides $l_1$.
Performing elementary row and column operations, we may assume that
\[
\f = \left[
\ba{cccc}
X & Y & 0 & 0 \\
0 & X & Y & 0 \\
0 & cZ & dZ & X
\ea
\right].
\]
Note that $X \{ Y^2, XY, X^2, dXZ, cYZ \} \subset U_{\f}$, hence, by Proposition \ref{2.2.1}, $c=0$ or $d=0$.
In the case when $d=0$ we obtain the fixed point $[\mu]$.
In the case when $c=0$ we obtain a matrix
\[
\boldsymbol{\nu} =
{\mathversion{bold}
\left[
\ba{cccc}
X & Y & 0 & 0 \\
0 & X & Y & 0 \\
0 & 0 & Z & X
\ea
\right]}
\]
that represents a $T$-fixed point in $\N$.
Considering types we see that $[\nu] \neq [\theta_{\sigma}]$, $[\iota_{\sigma}]$, $[\kappa_{\sigma}]$ for all $\sigma \in \Sym_3$.
Moreover, $[\nu] \neq [\lambda_{(Y,Z)}]$, $[\mu]$ because $U_{\nu} \neq U_{\lambda_{(Y,Z)}}$, $U_{\mu}$.
Thus, we obtain six new $T$-fixed points {\mathversion{bold}$[\nu_{\sigma}]$, $\sigma \in \Sym_3$}.

Assume now that $\Y$ has the ideal given at (ii).
We may write
\[
\f = \left[
\ba{cccc}
aY+bZ & X & 0 & 0 \\
Z & aY+bZ & X & 0 \\
\star & \star & \star & X
\ea
\right].
\]
Note that $X \{ X^2, XY, XZ, Y^2, b^2 Z^2 \} \subset U_{\f}$, hence, by Proposition \ref{2.2.1}, $b=0$
and, performing row and column operations, we may write
\[
\f = \left[
\ba{cccc}
aY & X & 0 & 0 \\
Z & aY & X & 0 \\
l_1 & l_2 & l_3 & X
\ea
\right]. \quad
\text{Moreover,} \quad
\f \sim \left[
\ba{cccc}
aY & X & 0 & 0 \\
Z & aY & X & 0 \\
0 & cY+dZ & 0 & X
\ea
\right]
\]
because $X$ divides $l_3$, since $X$ divides $\zeta_4$.
Since $ad XYZ \in U_{\f}$, we deduce that $d=0$ and
we obtain the fixed point $[\nu]$.

Assume, finally, that $\Y$ has the ideal given at (iii).
Thus, we may assume that
\[
\f = \left[
\ba{cccc}
X & cX+aY+bZ & 0 & 0 \\
Z & X & cX+aY+bZ & 0 \\
\star & \star & \star & X
\ea
\right].
\]
Notice that $X \{ X^2, Y^2, XY, YZ, b Z^2, cXZ \} \subset U_{\f}$ hence, by Proposition \ref{2.2.1},
$b=0$, $c=0$, and we may write
\[
\f = \left[
\ba{cccc}
X & aY & 0 & 0 \\
Z & X & aY & 0 \\
l_1 & l_2 & l_3 & X
\ea
\right].
\]
Since $X$ divides $\zeta_4$, $X$ also divides $Z l_3 - aY l_1$, from which, as in Section \ref{2.2.3}, it follows that
\[
\f \sim \left[
\ba{cccc}
X & aY & 0 & 0 \\
Z & X & Y & 0 \\
0 & Z & 0 & X
\ea
\right]. \quad \text{For $e \in \CC^*$ denote} \quad \psi_e = \left[
\ba{cccc}
X & aY & 0 & 0 \\
eZ & X & Y & 0 \\
0 & Z & 0 & X
\ea
\right].
\]
As in Section \ref{2.2.4}, the morphism ${\displaystyle \psi = \lim_{e \to 0} \psi_e}$ lies in $W^\ss$
and $\f \sim \psi$, which yields a contradiction.

\subsubsection{The case when the maximal minors of $\upsilon$ have a common linear factor}  
\label{2.2.7}

Assume that $\f \in W_1$ represents a $T$-fixed point
and that the maximal minors of $\upsilon$, denoted $\upsilon_1$, $\upsilon_2$, $\upsilon_3$,
have a common linear factor. Then
\[
\gcd(\upsilon_1, \upsilon_2, \upsilon_3) = l = \gcd(\zeta_1, \zeta_2, \zeta_3, \zeta_4).
\]
As before, we may assume that $l$ is a monomial, say $X$, the other cases
being obtained by a permutation of the variables.
Performing column operations, we may assume that the maximal minors of $\upsilon$ are $X^2$, $XY$, $XZ$.
It is now easy to see that we may write
\[
\f = \left[
\ba{cccc}
Y & X & 0 & 0 \\
Z & 0 & X & 0 \\
\star & \star & \star & l'
\ea
\right].
\]
Applying Proposition \ref{2.2.1} we can easily see that $l'$ is a monomial.
Indeed, if, say, $l' = aX + bY$ with $a \neq 0$, $b \neq 0$, then
$X \{ X^2, XY, XZ, Y^2, YZ \} \subset U_{\f}$.
In the sequel we will assume that $l' \in \{ X,Y,Z \}$.
When $l' = Y$ or $l' = Z$ we do not obtain any new fixed points in $\N$.
To see this it is enough to consider only the case when $l' =Y$,
the other case being obtained by swapping $Y$ and $Z$.
Thus, we consider the matrix
\[
\f = \left[
\ba{cccc}
Y & X & 0 & 0 \\
Z & 0 & X & 0 \\
aX+bZ & cZ & dZ & Y
\ea
\right].
\]
Note that $U_{\f} = \spann \{ X^2Y, XY^2, XYZ, aX^3, bX^2Z, dXZ^2 \}$, hence, in view of Proposition \ref{2.2.1},
precisely one among the numbers $a$, $b$, $d$ is non-zero.
We will first reduce the problem to the case when $c=0$.
Assume that $a$ and $c$ are non-zero, so
\[
\f = \left[
\ba{cccc}
Y & X & 0 & 0 \\
Z & 0 & X & 0 \\
aX & cZ & 0 & Y
\ea
\right]. \quad \text{For $e \in \CC^*$ denote} \quad \psi_e = \left[
\ba{cccc}
Y & X & 0 & 0 \\
Z & 0 & X & 0 \\
aX & ecZ & 0 & Y
\ea
\right].
\]
As at Section \ref{2.2.4}, the matrix ${\displaystyle \psi = \lim_{e \to 0} \psi_e}$ belongs to $W^\ss$
and $\f \sim \psi$, which is absurd.
The same argument will lead to a contradiction in the case when $b \neq 0$ and $c \neq 0$.
Assume now that $c \neq 0$, $d \neq 0$, so we may write
\[
\f = \left[
\ba{cccc}
Y & X & 0 & 0 \\
Z & 0 & X & 0 \\
0 & cZ & Z & Y
\ea
\right]. \quad \text{Consider the matrix} \quad \psi = \lim_{e \to 0}
\left[
\ba{cccc}
Y & X & 0 & 0 \\
Z & 0 & X & 0 \\
0 & ecZ & Z & Y
\ea
\right].
\]
Arguing as above we can show that $\f \sim \psi$.
It can be proven that this is impossible, however, for our purposes, all we need is to
observe that $\psi \sim \iota$, so, at any rate, we do not get a new $T$-fixed point.
We have thus reduced to the case when $c=0$.
Thus, $\f$ has one of the following forms:
\[
\left[
\ba{cccc}
Y & X & 0 & 0 \\
Z & 0 & X & 0 \\
X & 0 & 0 & Y
\ea
\right], \qquad \left[
\ba{cccc}
Y & X & 0 & 0 \\
Z & 0 & X & 0 \\
Z & 0 & 0 & Y
\ea
\right], \qquad \left[
\ba{cccc}
Y & X & 0 & 0 \\
Z & 0 & X & 0 \\
0 & 0 & Z & Y
\ea
\right].
\]
We obtain the fixed points $[\mu]$, $[\theta_{(Y,Z)}]$, $[\iota]$.

It remains to examine the case when $l'=X$.
We may write
\[
\f = \left[
\ba{cccc}
Y & X & 0 & 0 \\
Z & 0 & X & 0 \\
0 & aY+bZ & cY+dZ & X
\ea
\right].
\]
Note that $X \{ X^2, XY, XZ, aY^2, (b+c)YZ, dZ^2 \} \subset U_{\f}$,
hence, by Proposition \ref{2.2.1}, precisely one among the numbers $a$, $b+c$, $d$ is non-zero.
When $b+c \neq 0$ we get the fixed point represented by the matrix
\[
{\mathversion{bold}
\boldsymbol{\lambda (b : c)} = \left[
\ba{cccc}
Y & X & 0 & 0 \\
Z & 0 & X & 0 \\
0 & bZ & cY & X
\ea
\right].}
\]
We obtain a set $A$ of fixed points in $\N$ parametrised by $(b : c) \in \PP^1 \setminus \{ (1 : -1) \}$.
The above notation is justified because for $(b : c) = (1 : 0)$ we obtain the point $[\lambda]$
and for $(b : c) = (0 : 1)$ we obtain the point $[\lambda_{(Y,Z)}]$.
If $(b : c) = (1 : -1)$ we obtain the $T$-fixed point in $\N_2$ represented by the matrix
\[
\lambda (1 : -1) = \left[
\ba{cccc}
Y & X & 0 & 0 \\
Z & 0 & X & 0 \\
0 & Z & -Y & X
\ea
\right].
\]
In Section \ref{2.3} below we will see that this matrix is semi-stable.
The point $[\lambda (1 : -1)]$ lies in the closure of $A$, so we get a connected set
\[
\Lambda = \{ [\lambda(b : c)] \mid \ (b : c) \in \PP^1 \} \subset \N^T.
\]
Clearly, $\Lambda$ is not reduced to a point because it contains points from $\N_1$ and $\N_2$.
Our description of $\N^T$ (including Section \ref{2.3}) shows that $\dim(\N^T) \le 1$.
Thus, $\Lambda$ is a connected component of $\N^T$ of dimension $1$.
By \cite[Theorem 4.2]{carrell}, $\Lambda$ is smooth.
There is a morphism $\PP^1 \to \Lambda$ given by $(b : c) \mapsto [\lambda(b : c)]$,
hence $\Lambda$ is isomorphic to $\PP^1$, and hence $A$ is an affine line.
(It can be shown that the map $\PP^1 \to \Lambda$
is injective, hence it is an isomorphism, but we will not need these fact.)

Clearly, $\Lambda$ is fixed by the transposition $(Y, Z)$, so we get three projective lines
{\mathversion{bold} $\Lambda_{\sigma}$, $\sigma \in \A_3$,} of $T$-fixed points in $\N$.

Assume now that $a \neq 0$, $c=-b$ and $d=0$. We may write
\[
\f = \left[
\ba{cccc}
Y & X & 0 & 0 \\
Z & 0 & X & 0 \\
0 & Y+bZ & -bY & X
\ea
\right]. \quad \text{Denote} \quad \psi_e = \left[
\ba{cccc}
Y & X & 0 & 0 \\
Z & 0 & X & 0 \\
0 & eY+bZ & -bY & X
\ea
\right]
\]
for $e \in \CC^*$.
If $b=0$, then $[\f] = [\nu]$.
If $b \neq 0$, then, as before, ${\displaystyle \f \sim \lim_{e \to 0} \psi_e}$, hence $[\f] = [\lambda(1 : -1)]$.
Assume that $a=0$, $b=0$, $c=0$ and $d \neq 0$. Then $[\f] = [\nu_{(Y,Z)}]$.
Assume, finally, that $a=0$, $c=-b \neq 0$ and $d \neq 0$. Then, as before,
$[\f] = [\lambda(1:-1)]$.

\subsection{Fixed points in $\N_2$}      
\label{2.3}

Consider $\f \in W$. Assume that $\gcd(\zeta_1, \zeta_2, \zeta_3, \zeta_4) =q$ for some $q \in \Sym^2 V^*$.
According to \cite[Proposition 3.3.2]{illinois}, $\f$ is semi-stable if and only if $\f$ is equivalent to a morphism of the form
\[
\left[
\ba{cccc}
Y & X & 0 & l_1 \\
Z & 0 & X & l_2 \\
0 & Z & -Y & l_3
\ea
\right].
\]
Moreover, $U_{\f} = \spann \{ qX, qY, qZ \}$ and $q$ is a monomial.
Assume that $q = X^2$. Then $l_3 \neq 0 \mod (Y,Z)$ so, without loss of generality,
we may assume that
\[
\f = \left[
\ba{cccc}
Y & X & 0 & a_1X + b_1Y + c_1Z \\
Z & 0 & X & a_2X + b_2Y + c_2Z \\
0 & Z & -Y & X
\ea
\right].
\]
Since $\zeta_1$ is divisible by $X^2$ we deduce that $Y(b_2Y + c_2Z) - Z(b_1Y + c_1Z) = 0$, hence
\[
\f \sim \left[
\ba{cccc}
Y & X & 0 & a_1 X \\
Z & 0 & X & a_2 X \\
0 & Z & -Y & X
\ea
\right] \sim \left[
\ba{cccc}
Y & X & 0 & 0 \\
Z & 0 & X & 0 \\
0 & Z & -Y & X + a_2 Y - a_1 Z
\ea
\right].
\]
Since $q=X(X + a_2 Y - a_1 Z)$, we deduce that $a_1 = 0$, $a_2 = 0$
and we obtain the $T$-fixed point $[\lambda(1 : -1)]$ from Section \ref{2.2.7}.
Analogously, when $q=Y^2$ or $Z^2$, we get the $T$-fixed points $[\lambda(1 : -1)_{\sigma}]$
for $\sigma = (X, Y)$, respectively, $(X, Z)$.

Assume now that $q=XY$. Clearly, $l_1 = 0 \mod (X,Y)$, hence we may write
\[
\f = \left[
\ba{cccc}
Y & X & 0 & 0 \\
Z & 0 & X & l_2 \\
0 & Z & -Y & l_3
\ea
\right].
\]
Since $X$ divides $X l_3 + Y l_2$, we see that $X$ divides $l_2$.
Since $Y$ divides $X l_3 + Y l_2$, we see that $Y$ divides $l_3$.
Thus, $\f$ is equivalent to the matrix
\[
{\mathversion{bold}
\boldsymbol{\xi} = \left[
\ba{cccc}
Y & X & 0 & 0 \\
Z & 0 & X & 0 \\
0 & Z & -Y & Y
\ea
\right]}
\]
giving a $T$-fixed point in $\N$.
We get three new isolated $T$-fixed points {\mathversion{bold} $[\xi_{\sigma}]$, $\sigma \in \A_3$.}


\section{Proof of Theorem 1} 
\label{section_3}

In Section \ref{section_2} we found that the isolated points of $\N^T$ are
$\alpha$, $\beta$, $\gamma$, $\delta$, $\varepsilon$, $\zeta$, $\eta$, $\theta$, $\iota$, $\kappa$,
$\mu$, $\nu$, $\xi$ and the points obtained from these by permutations of variables.
Thus, $\N^T$ consists of $62$ isolated points and $3$ projective lines, namely $\Lambda$
and two other lines obtained by permutations of variables.
In Section \ref{3.3} below we will examine the action of $T$ on the tangent spaces at the $T$-fixed points.
Before that, we will find the initial moduli spaces of height zero and dimension $12$.

\subsection{Diophantine equations}
\label{3.1}

The condition that the Kronecker moduli space $\N(3r, m, n)$ have dimension $12$ is equivalent to the equation
\[
\tag{3.1.1}
\label{3.1.1}
3rmn- m^2 - n^2 =11, \qquad r, m, n \in \NN.
\]
We fix $r$ and solve for $(m, n)$.
The set of solutions is preserved under the operations
\[
R, S \colon \ZZ^2 \to \ZZ^2, \qquad R(m,n) = (n, 3rn - m), \quad S(m,n) = (n,m).
\]
We say that two solutions $(m,n)$ and $(m',n')$ are equivalent
if one of them can be obtained from the other by applying finitely many times
the operations $R$ and $S$.
Let $(m,n)$ be a solution that is smallest in its equivalence class,
in the sense that $n \le n'$, $n \le m'$ for all $(m',n') \sim (m,n)$ and
$m \le m'$ for all $(m',n) \sim (m,n)$.
Thus, $n \le m \le 3rn/2$. Indeed, if $3rn/2 < m$, then
$(m', n) = (3rn-m, n)$ would be an equivalent solution for which $m' < m$,
contradicting the choice of $(m,n)$. Checking the cases when $m=n$ or $m=3rn/2$
yields no solutions, so we may assume that $n < m < 3rn/2$. Thus,
\[
11 > 3rmn - 2 m^2 = (3rn-2m) m \quad \text{forcing} \quad 2 \le m \le 10.
\]
Checking the cases when $2m/(3r) < n < m \le 10$ yields the solution $(r,m,n) = (1,4,3)$.
We conclude that the solutions to equation \ref{3.1.1} are of the form $(1,m,n)$, where $(m,n) \sim (4,3)$.
As a consequence, the moduli spaces of height zero and dimension $12$ are isomorphic to $\N(3, 4, 3)$.
Note that the exceptional bundle $E$ is a line bundle.

Analogously, the solutions to the diophantine equation
\[
3rmn- m^2 - n^2 =10, \qquad r, m, n \in \NN,
\]
are of the form $(4, m, n)$, where $(m, n) \sim (1, 1)$.
It follows that the moduli spaces of height zero and dimension $11$
are associated to exceptional bundles of rank $4$.
According to \cite{rudakov}, there are no exceptional bundles of rank $4$ on $\PP^2$.
We conclude that there are no moduli spaces of height zero and dimension $11$.

\subsection{Initial moduli spaces}
\label{3.2}

Let us determine the initial moduli spaces $\M$ of height zero and dimension $d=12$.
We saw in Section \ref{3.1} that the exceptional bundle $E$ is a line bundle.
Without loss of generality, we may assume that $E = \O_{\PP^2}$.
We first consider the case when $\mu \le \mu(E)$, the case when $\mu \ge \mu(E)$ being dual.
The condition that $\M$ be initial implies the inequalities
\[
\mu(E) - \frac{1}{3 \rank(E)^2} < \mu \le \mu(E), \quad \text{that is}, \quad - \frac{1}{3} < \mu \le 0.
\]
By \cite[Proposition 30]{drezet_reine}, we have the inequalities
\[
\rank(E) \sqrt{d-1} \le r < 3\rank(E)^2 \sqrt{d-1},
\]
forcing $4 \le r \le 9$.
The condition that $\M$ have height zero is equivalent to the condition
\[
\Delta = \frac{1}{2} (\mu - \mu(E)) (\mu - \mu(E) + 3) + 1 - \Delta(E)
= \frac{1}{2} \mu (\mu + 3) + 1.
\]
The isomorphism $\M \isom \N(3, 4, 3)$ of \cite{drezet_reine} preserves stable points,
hence $\M$ contains stable points.
This leads to the formula $d-1 = r^2 (2\Delta -1)$ of \cite[Corollary 14.5.4]{lepotier}.
Combining with the equation above we obtain the diophantine equation
\[
11 = c_1^2 + 3 r c_1^{} + r^2.
\]
Verifying this equation for $r = 4, \ldots, 9$ and $c_1 = [- r/3] + 1, \ldots, 0$
yields the solutions $(r, c_1) = (5, -1)$ and $(7, -2)$.
We obtain the moduli spaces $\M(5, -1, 4)$ and $\M(7, -2, 6)$.

By duality, if $\mu \ge \mu(E)$, we obtain the initial moduli spaces $\M(5, 1, 4)$ and $\M(7, 2, 6)$.

\subsection{Torus representation of the tangent spaces}
\label{3.3}

Let $\f \in W^\ss$ give a torus fixed point in $\N(3,4,3)$.
Assume that there are morphisms of groups
\[
u \colon (\CC^*)^3 \lra \GL(4, \CC), \qquad v \colon (\CC^*)^3 \lra \GL(3,\CC),
\]
\[
u = \left[
\ba{cccc}
u_1 & 0 & 0 & 0 \\
0 & u_2 & 0 & 0 \\
0 & 0 & u_3 & 0 \\
0 & 0 & 0 & u_4
\ea
\right], \qquad v = \left[
\ba{ccc}
v_1 & 0 & 0 \\
0 & v_2 & 0 \\
0 & 0 & v_3
\ea
\right],
\]
such that $t \f = v(t) \f u(t)$ for all $t \in (\CC^*)^3$.
In Table 1 below we give such morphisms $u$ and $v$ for all fixed points $\f$ found in Section \ref{section_2}.
According to \cite[Formula (6.1.1)]{choi_maican}, the action of $(\CC^*)^3$ on $\T_{[\f]} \N$, denoted by $\star$,
is given by
\[
\tag{3.3.1}
\label{3.3.1}
t \star [w] = [ v(t)^{-1} (t w) u(t)^{-1}]
\]
and is induced by an action on $W$ given by the same formula.
According to \cite[Formula (6.1.3)]{choi_maican}, the induced action on $\T_{\f} (G \f)$ is given by
\[
\tag{3.3.2}
\label{3.3.2}
t \star (A, B) = (u(t) A \, u(t)^{-1}, v(t)^{-1} B \, v(t)).
\]
Here we identify $\T_{\f} (G \f)$ with the tangent space of $G$ at its neutral element
\[
\T_e G = (\End(4\O(-2)) \oplus \End(3\O(-1)))/\CC
\]
whose vectors are represented by pairs $(A, B)$ of matrices.
Both actions factor through $T$.
The list of weights for the action of $T$ on $\T_{[\f]} \N(3,4,3)$, denoted by $\chi^*[\f]$,
is obtained by subtracting the list of weights for the action of $T$ on $\T_{\f} (G \f)$,
obtained by means of (\ref{3.3.2}), from the list of weights
for the action of $T$ on $\T_{\f} W$, obtained using (\ref{3.3.1}):
\begin{multline*}
\chi^*[\f] = \{  v_j^{-1} u_i^{-1} t_k^{} \mid i=1,2,3,4,\ j = 1,2,3,\ k = 0,1,2 \} \\
\setminus
\big( \{ u_i^{} u_j^{-1} \mid i,j = 1,2,3,4 \} \cup \{ v_i^{-1} v_j^{} \mid i,j = 1,2,3 \} \setminus \{ \chi_0 \} \big).
\end{multline*}
Here $\chi_0$ is the trivial character of $T$.
It is convenient to use additive notation when dealing with characters.
We replace the character $t_0^i t_1^j t_2^k$ with the expression $ix+jy+kz$, where $i, j, k \in \ZZ$.
In Table 2 below we give, in additive notation, the list of weights for each fixed point from Section \ref{section_2}.
These lists are obtained with the aid of the {\sc Singular} \cite{singular} program from Appendix \ref{appendix_A}.
The points $\lambda (b : c)$ from Section \ref{2.2.7} are ignored because
the $T$-representation of the tangent space at a point is unchanged if the point varies in
a connected component of $\N^T$,
so it is enough to examine only one point on $\Lambda$, say $[\lambda]$.

\begin{table}[!hpt]{Table 1}
\begin{center}
\begin{tabular}{| c | l | l |}
\hline
Fixed point & $(v_1, v_2, v_3)$ & $(u_1, u_2, u_3, u_4)$ \\
\hline \hline
$\alpha$ & $(z, x, y)$ & $(0, x-z, y-x, z-y)$ \\
\hline
$\beta$ & $(y, x, x)$ & $(0, x-y, y-x, z-x)$ \\
\hline
$\gamma$ & $(y, x, 2y-x)$ & $(0, x-y, y-x, x-2y+z)$ \\
\hline
$\delta$ & $(y, 2y-z, 2y-x)$ & $(0, z-y, x-2y+z, 2x-2y)$ \\
\hline
$\varepsilon$ & $(x, z, -x+y+z)$ & $(0, y-x, x-z, 2x-y-z)$ \\
\hline
$\zeta$ & $(y, x, -x+y+z)$ & $(0, x-y, y-x, 2x-y-z)$ \\
\hline
$\eta$ & $(y,x,2x-y)$ & $(0, x-y, y-x, 2y-2x)$ \\
\hline
$\theta$ & $(x, x, x-y+z)$ & $(0, y-x, z-x, y-z)$ \\
\hline
$\iota$ & $(y, x+y-z, x+2y-2z)$ & $(0, z-y, -x-y+2z, -2y+2z)$ \\
\hline
$\kappa$ & $(x, x-y+z, x-2y+2z)$ & $(0, y-x, -x+2y-z, 2y-2z)$ \\
\hline
$\lambda$ & $(y, -x+y+z, z)$ & $(0, x-y, 2x-y-z, x-z)$ \\
\hline
$\mu$ & $(x, z, y)$ & $(y-x, 0, x-z, x-y)$ \\
\hline
$\nu$ & $(x, 2x-y, 2x-2y+z)$ & $(0, y-x, 2y-2x, -x+2y-z)$ \\
\hline
$\xi$ & $(y, z, -x+y+z)$ & $(0, x-y, x-z, x-z)$ \\
\hline
\end{tabular}
\end{center}
\end{table}

\begin{table}[!hpt]{Table 2}
\begin{center}
\begin{tabular}{| c | l |}
\hline
\begin{tabular}{l} Fixed \\ point $\f$ \end{tabular} & \begin{tabular}{l} $\chi^*[\f]$ \end{tabular} \\
\hline \hline
$\alpha$
&
\begin{tabular}{l}
$2y-2z, -x+z, -2x+2z, y-z, -x+y, x-y$, \\
$-y+z, -x+z, 2x-2y, x-y, x-z, y-z$
\end{tabular} \\
\hline
$\beta$
&
\begin{tabular}{l}
$-x+z, -x+z, -x+y, -2x+2y, -2x+y+z$, \\
$x-2y+z, -y+z, -y+z, x-y, x-z, x-z, y-z$
\end{tabular} \\
\hline
$\gamma$
&
\begin{tabular}{l}
$-x+z, -x+z, -2x+y+z, -y+z, x-2y+z, x-y$, \\
$-y+z, 3x-3y, 2x-2y, -x+2y-z, x-z, y-z$
\end{tabular} \\
\hline
$\delta$
&
\begin{tabular}{l}
$-2y+2z, x-2y+z, x-y, -y+z, x-y, x-z$, \\
$y-z, -x+y, -2x+y+z, -2x+2z, -x+y, -x+z$
\end{tabular} \\
\hline
$\varepsilon$
&
\begin{tabular}{l}
$y-z, -y+z, x-z, 2x-2y, -2x+2z, -x+z$, \\
$x-y, -y+z, -3x+y+2z, -2x+2y, -x+y, -x+z$
\end{tabular} \\
\hline
$\zeta$
&
\begin{tabular}{l}
$-y+z, y-z, x-2y+z, x-y, -y+z, 2x-2y, -x+z$, \\
$-2x+2z, -2x+y+z, -3x+y+2z, -x+y, -x+z$
\end{tabular} \\
\hline
$\eta$
&
\begin{tabular}{l}
$-y+z, -x+z, -2x+y+z, -x+z, -2x+y+z, -3x+2y+z$, \\
$x-2y+z, -y+z, -x+z, 2x-3y+z, x-2y+z, -y+z$
\end{tabular} \\
\hline
$\theta$
&
\begin{tabular}{l}
$-x+y, -x+2y-z, -x+y, x-y, x-z, 2y-2z$, \\
$y-z, -x+z, -x+z, -x-y+2z, -x+y, -x+z$
\end{tabular} \\
\hline
$\iota$
&
\begin{tabular}{l}
$-x-2y+3z, -x+z, -x-y+2z, 2x-2z, x-z, x-y$, \\
$-y+z, y-z, -x+2y-z, -x+y, -x+y, -x+z$
\end{tabular} \\
\hline
$\kappa$
&
\begin{tabular}{l}
$-x+y, -x+3y-2z, -x+2y-z, x-y, x-z, y-z$, \\
$-x-2y+3z, -y+z, -x+z, -x-y+2z, -x+y, -x+z$
\end{tabular} \\
\hline
$\lambda$
&
\begin{tabular}{l}
$y-z, -x+2y-z, -2x+2z, -x+y, -2x+2y, -2x+y+z$, \\
$-x+z, -x-y+2z, -y+z, 0, -x+y, -x+z$
\end{tabular} \\
\hline
$\mu$
&
\begin{tabular}{l}
$x-2y+z, x-y, -2x+2z, -x+z, -y+z, -x+z$, \\
$-x-y+2z, -2x+y+z, -x+2y-z, -x+y, -x+y, -x+z$
\end{tabular} \\
\hline
$\nu$
&
\begin{tabular}{l}
$-2x+y+z, -2x+3y-z, -2x+2y, -x+z, -x+y, -y+z$, \\
$-2y+2z, -y+z, -x+z, -x-y+2z, -x+y, -x+z$
\end{tabular} \\
\hline
$\xi$
&
\begin{tabular}{l}
$2x-y-z, -x+2y-z, -x-y+2z, -x+y, x-y, -y+z$, \\
$-x+z, -x-y+2z, -x+y, -x+z, x-y, -y+z$
\end{tabular} \\
\hline
\end{tabular}
\end{center}
\end{table}

Let $\lambda(\tau) = (\tau^{n_0}, \tau^{n_1}, \tau^{n_2})$ be a one-parameter subgroup of $T$
that is not orthogonal to any non-zero character appearing in Table 2.
The set of weights from Table 2 is contained on the set
\[
\{ ix + jy + kz \mid -3 \le i, j, k \le 3 \},
\]
so we can choose $\lambda(\tau) = (1, \tau, \tau^4)$.
For each $T$-fixed point $[\f]$ in $\N$ denote by $p[\f]$ the number of characters $\chi \in \chi^*[\f]$
satisfying the condition $\langle \lambda, \chi \rangle > 0$.
Using the procedure ``positive-parts'' from Appendix \ref{appendix_A}
we compute the list of numbers $p[\f_{\sigma}]$, $\sigma \in \Sym_3$.
The results are written Table 3 below.

\begin{table}[!hpt]{Table 3}
\begin{center}
\begin{tabular}{| c | l | c | c | l |}
\cline{1-2} \cline{4-5}
Fixed point $\f$ & $p[\f_{\sigma}]$, $\sigma \in \Sym_3$
& &
Fixed point $\f$ & $p[\f_{\sigma}]$, $\sigma \in \Sym_3$
\\
\cline{1-2} \cline{4-5}
$\alpha$ & $5, 7$
& &
$\beta$ & $8, 6, 6, 7, 5, 6$
\\
\cline{1-2} \cline{4-5}
$\gamma$ & $6, 4, 7, 9, 6, 5$
& &
$\delta$ & $8, 5, 6$
\\
\cline{1-2} \cline{4-5}
$\varepsilon$ & $8, 5, 5, 8, 4, 7$
& &
$\zeta$ & $9, 4, 5, 10, 4, 7$
\\
\cline{1-2} \cline{4-5}
$\eta$ & $12, 3, 6$
& &
$\theta$ & $7, 7, 2, 5, 4, 9$
\\
\cline{1-2} \cline{4-5}
$\iota$ & $7, 6, 4, 6, 4, 6$
& &
$\kappa$ & $7, 6, 3$
\\
\cline{1-2} \cline{4-5}
$\lambda$ & $9, 6, 1$
& &
$\mu$ & $10, 4, 3, 9, 2, 8$
\\
\cline{1-2} \cline{4-5}
$\nu$ & $11, 5, 3, 8, 0, 8$
& &
$\xi$ & $8, 3, 5$
\\
\cline{1-2} \cline{4-5}
\end{tabular}
\end{center}
\end{table}

We quote below the Homology Basis Formula \cite[Theorem 4.4]{carrell}.
Let $X_1, \ldots, X_m$ denote the irreducible components of $\N^T$.
Denote $p(i) = p(X_i) = p [\f]$ for some (or, in fact, any) point $[\f] \in X_i$.
Then for $0 \le n \le 2 \dim (\N)$ we have the isomorphism
\[
\tag{3.3.3}
\label{3.3.3}
\H_n (\N, \ZZ) \isom \bigoplus_{1 \le i \le m} \H_{n-2 p(i)} (X_i, \ZZ).
\]
Let $\Pi$ denote the set of isolated $T$-fixed points in $\N$.
From (\ref{3.3.3}) we get the formula
\[
P_{\N}(x) = \sum_{[\f] \in \Pi} x^{2 p[\f]} + \sum_{\sigma \in \A_3} (x^2 + 1) x^{2 p(\Lambda_{\sigma})}.
\]
Substituting the values of $p[\f]$ from Table 3 yields the expression for the Poincar\'e polynomial
from Theorem \ref{theorem_1}.

According to \cite[Section 4.2.8]{carrell}, Formula (\ref{3.3.3}) respects the Hodge decomposition, that is,
for $0 \le p, q \le \dim(\N)$, we have the isomorphism
\[
\tag{3.3.4}
\label{3.3.4}
\H^p (\N, \Omega_{\N}^q) \isom \bigoplus_{1 \le i \le m} \H^{p - p(i)} \big(X_i, \Omega_{X_i}^{q - p(i)}\big).
\]
This shows that $h^{pq}(\N) = 0$ for $p \neq q$, because the same is true of the Hodge numbers of all $X_i$
(which are points or projective lines).


\section{The torus fixed locus of $\M_{\PP^2}(5,1)$}
\label{section_4}

For the convenience of the reader we recall from \cite{illinois} the classification of semi-stable sheaves
on $\PP^2$ having Hilbert polynomial $5m+1$.
In $\M_{\PP^2}(5,1)$ we have four smooth strata $\M_0$, $\M_1$, $\M_2$, $\M_3$.
The stratum $\M_0$ is open and consists of sheaves having a presentation of the form
\[
0 \lra 4\O(-2) \stackrel{\f}{\lra} 3\O(-1) \oplus \O \lra \F \lra 0,
\]
where $\f_{11}$ is semi-stable as a Kronecker module.
We denote by $\M_{01} \subset \M_0$ the locally closed subset
given by the condition that the greatest common divisor of the maximal minors of $\f_{11}$,
denoted $\zeta_1$, $\zeta_2$, $\zeta_3$, $\zeta_4$, is a linear form.
Likewise, $\M_{02}$ is the locally closed subset given by the
condition that $\zeta_1$, $\zeta_2$, $\zeta_3$, $\zeta_4$ have a common factor of degree $2$.
The complement $\M_0 \setminus (\M_{01} \cup \M_{02})$ is given by the condition that
$\zeta_1$, $\zeta_2$, $\zeta_3$, $\zeta_4$ have no common factor.
In this case the zero-set of $\zeta_1$, $\zeta_2$, $\zeta_3$, $\zeta_4$
is a zero-dimensional subscheme $\Z \subset \PP^2$ of length $6$ that is not contained in a conic curve.
Thus, $\M_0 \setminus (\M_{01} \cup \M_{02})$ consists of all sheaves of the form
$\O_Q(-\Z)(2)^\D$, where $Q \subset \PP^2$ is a quintic curve, $\Z$ is a subscheme of $Q$ as above
and $\O_Q(-\Z) \subset \O_Q$ is its ideal sheaf. Here $Q$ is defined by the equation $\det(\f) = 0$.
For a one-dimensional sheaf $\F$ on $\PP^2$
we use the notation $\F^\D$ to denote the dual sheaf ${\mathcal Ext}^1(\F, \omega_{\PP^2})$.

The stratum $\M_1$ is locally closed and has codimension $2$.
It consists of those sheaves given by exact sequences of the form
\[
0 \lra \O(-3) \oplus \O(-2) \stackrel{\f}{\lra} 2\O \lra \F \lra 0,
\]
where $\f_{12}$ and $\f_{22}$ are linearly independent.
Let $\M_{10} \subset \M_1$ be the open subset (in the relative topology) given by the condition that $\f_{12}$ and $\f_{22}$
have no common factor. Clearly, $\M_{10}$ consists of sheaves of the form $\O_Q(-\X)(2)$,
where $Q \subset \PP^2$ is a quintic curve,
$\X$ is the intersection of two conic curves without common component, $\X$ is contained in $Q$ and $\O_Q(-\X) \subset \O_Q$
is its ideal sheaf. The complement $\M_{11} = \M_1 \setminus \M_{10}$ is given by the condition that
$\f_{12} = l l_1$, $\f_{22} = l l_2$ for some linear forms $l$, $l_1$, $l_2$.

The stratum $\M_2$ is locally closed of codimension $3$. The points $[\F]$ in $\M_2$ are given by
exact sequences of the form
\[
0 \lra \O(-3) \oplus \O(-2) \oplus \O(-1) \stackrel{\f}{\lra} \O(-1) \oplus 2\O \lra \F \lra 0,
\]
where $\f_{23}$ has linearly independent entries,
$\f_{13} = 0$, $\f_{12} \neq 0$, and $\f_{11}$ is not divisible by $\f_{12}$.

The deepest stratum $\M_3$ is closed of codimension $5$ and is isomorphic to the universal quintic curve.
The sheaves $\F$ giving points in $\M_3$ are cokernels of the form
\[
0 \lra 2\O(-3) \stackrel{\f}{\lra} \O(-2) \oplus \O(1) \lra \F \lra 0,
\]
where $\f_{11}$ and $\f_{12}$ are linearly independent. Equivalently, these are the sheaves of the form
$\O_Q(-P)(1)^\D$, where $P$ is a closed point on a quintic curve $Q$.

Let $W_i$ be the set of morphisms $\f$ as above such that $\Coker(\f)$ gives a point in $\M_i$.
The ambient vector space $\WW_i$ of morphisms of sheaves is acted upon by the group
of automorphisms $G_i$.
Thus
\[
\WW_1 = \Hom(4\O(-2), 3\O(-1) \oplus \O),
\]
\[
G_1 = (\Aut(4\O(-2)) \times \Aut(3\O(-1) \oplus \O))/\CC^*
\]
etc.
Clearly, $W_i$ is $G_i$-invariant.
As shown in \cite{illinois}, the canonical maps $W_i \to \M_i$ are geometric quotient maps.
In particular, the strata $\M_i$ are smooth.


\subsection{Fixed points in $\M_0$}
\label{4.1}

Given a morphism $\psi \colon 4\O(-2) \to 3\O(-1)$ and a monomial $d$ of degree $5$
belonging to the ideal generated by the maximal minors of $\psi$,
we denote by $\M(\psi, d)$ the image in $\M_{\PP^2}(5,1)$
of the set of morphisms $\f \in W_0$ for which $\f_{11} = \psi$ and $\det(\f) = d$.
If $[\Coker(\f)]$ is a $T$-fixed point in $\M_0$, then, obviously, $\f_{11}$ gives a $T$-fixed point in $\N(3, 4, 3)$.
The torus-fixed points in $\N(3, 4, 3)$ have been classified in Section \ref{section_2}.

Consider first the torus action on $\M_0 \setminus (\M_{01} \cup \M_{02})$.
Clearly, $\O_Q(-\Z)(2)^\D$ gives a $T$-fixed point precisely if $Q$ and $\Z$ are $T$-invariant.
Up to a permutation of variables, there are seven schemes $\Z$ corresponding to the
matrices $\alpha$, $\beta$, $\gamma$, $\delta$, $\varepsilon$, $\zeta$, and $\eta$ from Section \ref{section_2}.
For each fixed $\Z$ there are twelve invariant quintics $Q$ containing it.
In the sequel, we will concentrate on finding the fixed points in $\M_{01} \cup \M_{02}$.
Thus, $\f_{11}$ is equivalent to one of the matrices
$\theta$, $\iota$, $\kappa$, $\lambda(a : b)$, $\mu$, $\nu$, and $\xi$ from Section \ref{section_2}.
Among these, $\lambda(1 :  -1)$ and $\xi$ correspond to points in $\M_{02}$,
while the other matrices correspond to points in $\M_{01}$.
We will examine only the case when $\f_{11}$ is equivalent to $\theta$, the other cases being analogous.
Write $d = \det(\f) = X^i Y^j Z^k$ and fix quadratic forms $q_1, q_2, q_3, q_4$ satisfying the equation
\[
\left[
\ba{cccc}
q_1 & q_2 & q_3 & q_4
\ea
\right] \Theta = d, \quad \text{where} \quad \Theta = \left[
\ba{cccc}
-XYZ & X^2 Z & X^2 Y & -XZ^2
\ea
\right]^\TR
\]
is the column vector of maximal minors of $\theta$.
The set $\M(\theta, d)$ is parametrised by morphisms of the form
\[
\left[
\ba{cccc}
X & Y & 0 & 0 \\
X & 0 & Z & 0 \\
0 & Z & 0 & X \\
q_1' & q_2' & q_3' & q_4'
\ea
\right],
\]
where
\[
\left[
\ba{cccc}
q_1' & q_2' & q_3' & q_4'
\ea
\right] \Theta = d.
\]
From the relation
\[
\left[
\ba{cccc}
q_1' - q_1 & q_2' - q_2 & q_3' - q_3 & q_4' - q_4
\ea
\right] \Theta = 0
\]
we deduce that the row matrix above is a linear combination (with polynomial coefficients)
of the rows of the matrix
\[
\left[
\ba{cccc}
X & Y & 0 & 0 \\
X & 0 & Z & 0 \\
0 & Z & 0 & X \\
Z & 0 & 0 & -Y
\ea
\right]
\]
obtained by adjoining a row to $\theta$.
It follows that $\M(\theta, d)$ is parametrised by the set $A$ of morphisms of the form
\[
\f(u) = \left[
\ba{cccc}
X & Y & 0 & 0 \\
X & 0 & Z & 0 \\
0 & Z & 0 & X \\
q_1 + uZ & q_2 & q_3 & q_4 - uY
\ea
\right], \quad u = aY + bZ, \quad a, b \in \CC.
\]
The canonical map $A \to \M(\psi, d)$ is an isomorphism, which can be seen using an argument
similar to that of \cite[Proposition 5.1]{choi_maican}.
The induced action of $T$ on $A$ is given by $(t, u) \mapsto t_0^{1-i} t_1^{1-j} t_2^{2-k} (tu)$.
Choosing coordinates $(a, b)$ we identify $A$ with $\AA^2$.
The induced action of $T$ on $\AA^2$ is given by
\[
t (a, b) = (t_0^{1-i} t_1^{2-j} t_2^{2-k} a, \ t_0^{1-i} t_1^{1-j} t_2^{3-k} b).
\]
We get an isolated fixed point, namely $(0, 0)$, unless $(i, j, k) = (1, 2, 2)$ or $(1, 1, 3)$,
in which case we get an affine line of fixed points.
Summarising, we obtain the following proposition.

\begin{proposition}
\label{4.1.1}
Assume that $\psi$ gives a $T$-fixed point in $\N(3, 4, 3)$.
Let $\zeta_1$, $\zeta_2$, $\zeta_3$, $\zeta_4$ be its maximal minors.
Then, for any monomial $d$ of degree $5$ belonging to the ideal $(\zeta_1, \zeta_2, \zeta_3, \zeta_4)$,
the set of fixed points for the action of $T$ on $\M(\psi, d)$ has precisely one irreducible component,
which is either a point or an affine line.
We have a line in the following cases:
\begin{center}
\begin{tabular}{lr}
$\psi$ & $d$ \\
\hline
$\theta$ & $XYZ^3$, $XY^2 Z^2$ \\
$\iota$ & $XY^3 Z$, $XY^2 Z^2$ \\
$\kappa$ & $XY^2 Z^2$ \\
$\lambda(a : b)$ & $X^2 Y^2 Z$, $X^2 YZ^2$ \\
$\mu$ & $XY^3 Z$, $XY^2 Z^2$ \\
$\nu$ & $X^2 Y^2 Z$, $X^2 YZ^2$ \\
$\xi$ & $X^3 YZ$, $XY^3 Z$, $XYZ^3$, $XY^2 Z^2$, $X^2 YZ^2$
\end{tabular}
\end{center}
\end{proposition}

\noindent
Note that the torus fixed lines in $\M(\lambda(a : b), X^2 Y^2 Z)$ sweep a surface isomorphic to $\PP^1 \times \AA^1$.
Thus, $\M_0^T$ has six irreducible components of dimension $2$ obtained from the surface $\Sigma_0$ with parametrisation
\[
\lambda((a : b), c) = \left[
\ba{cccc}
Y & X & 0 & 0 \\
Z & 0 & X & 0 \\
0 & aZ & bY & X \\
0 & (1-c)YZ & c Y^2 & 0
\ea
\right], \qquad (a : b) \in \PP^1, \quad c \in \AA^1,
\]
by permutations of variables.


\subsection{Fixed points in $\M_1$}
\label{4.2}

Assume that the point in $\M_1$ represented by the morphism
\[
\f = \left[
\ba{cc}
f_1 & q_1 \\
f_2 & q_2
\ea
\right]
\]
is fixed by $T$. Then, since the fibres of the map $W_1 \to \M_1$ are the $G_1$-orbits,
we deduce that for each $t \in T$ there is $(g(t), h(t)) \in G_1$ such that $t \f = h(t) \f g(t)$.
By the argument at \cite[Section 2.1.1]{homology}, we deduce that $q_1$ and $q_2$ are distinct monomials of degree $2$.
Moreover, $d = \det(\f)$ is a monomial of degree $5$ that varies in the ideal $(q_1, q_2)$.
We denote by $\M(q_1, q_2, d)$ the image in $\M_{\PP^2}(5,1)$ of the set of morphisms $\f \in W_1$
for which $\f_{12} = q_1$, $\f_{22} = q_2$ and $\det(\f) = d = X^i Y^j Z^k$.
If $q_1$ and $q_2$ have no common factor,
then, $\M(q_1, q_2, d)$ consists of a single $T$-fixed point of the form $\O_Q(-\X)(2)$.
Assume next that $q_1 = l l_1$, $q_2 = l l_2$ for some linear forms $l, l_1, l_2 \in \{ X, Y, Z \}$.
Fix monomials $f_1$ and $f_2$ such that $f_1 q_2 - f_2 q_1 = d$.
It is easy to see that $\M(l l_1, l l_2, d)$ is parametrised by the set $A$ of morphisms of the form
\[
\left[
\ba{cc}
f_1 + q l_1 & l l_1 \\
f_2 + q l_2 & l l_2
\ea
\right]
\]
with $q$ a quadratic form in the two variables different from $l$.
Up to a permutation of variables, there are only two cases to be considered: $(q_1, q_2) = (X^2, XY)$ or $(XZ, YZ)$.
In the first case
\[
\f = \left[
\ba{cc}
f_1 + q X & X^2 \\
f_2 + q Y & XY
\ea
\right], \qquad q = a Y^2 + b Z^2 + c YZ.
\]
Choosing coordinates $(a, b, c)$, we identify $A$ with $\AA^3$.
We have
\[
t \f = \left[
\ba{cc}
t_0^{i-1} t_1^{j-1} t_2^k f_1^{} + (t q) t_0^{} X & t_0^2 X^2 \\
t_0^{i-2} t_1^j t_2^k f_2^{} + (t q) t_1^{} Y & t_0^{} t_1^{} XY
\ea
\right] \sim \left[
\ba{cc}
f_1 + t_0^{2-i} t_1^{1-j} t_2^{-k} (t q) X & X^2 \\
f_2 + t_0^{2-i} t_1^{1-j} t_2^{-k} (t q) Y & XY
\ea
\right].
\]
The induced action of $T$ on $\AA^3$ is given by
\[
t (a, b, c) = (t_0^{2-i} t_1^{3-j} t_2^{-k} a, \ t_0^{2-i} t_1^{1-j} t_2^{2-k} b, \ t_0^{2-i} t_1^{2-j} t_2^{1-k} c).
\]
We get an isolated fixed point unless $(i, j, k) = (2, 3, 0)$, or $(2, 1, 2)$, or $(2, 2, 1)$, in which case we get
an affine line of fixed points.
Assume now that
\[
\f = \left[
\ba{cc}
f_1 + q X & XZ \\
f_2 + q Y & YZ
\ea
\right], \qquad q = a X^2 + b Y^2 + c XY.
\]
We have
\[
t \f = \left[
\ba{cc}
t_0^i t_1^{j-1} t_2^{k-1} f_1^{} + (t q) t_0^{} X & t_0^{} t_2^{} XZ \\
t_0^{i-1} t_1^j t_2^{k-1} f_2^{} + (t q) t_1^{} Y & t_1^{} t_2^{} YZ
\ea
\right] \sim \left[
\ba{cc}
f_1^{} + t_0^{1-i} t_1^{1-j} t_2^{1-k} (t q) X & XZ \\
f_2^{} + t_0^{1-i} t_1^{1-j} t_2^{1-k} (t q) Y & YZ
\ea
\right].
\]
The induced action of $T$ on $A \isom \AA^3$ is given by
\[
t (a, b, c) = (t_0^{3-i} t_1^{1-j} t_2^{1-k} a, \ t_0^{1-i} t_1^{3-j} t_2^{1-k} b, \ t_0^{2-i} t_1^{2-j} t_2^{1-k} c).
\]
We get an isolated fixed point unless $(i, j, k) = (3,1,1)$, or $(1,3,1)$, or $(2,2,1)$, in which case
we get an affine line of fixed points. Summarising, we obtain the following proposition.

\begin{proposition}
\label{4.2.1}
Let $q_1$ and $q_2$ be distinct monomials of degree $2$. Let $d$ be a monomial of degree $5$
in the ideal $(q_1, q_2)$. Then the set of fixed points for the action of $T$ on $\M(q_1, q_2, d)$
has precisely one irreducible component, which is either a point or an affine line.
\end{proposition}


\subsection{Fixed points in $\M_2$}
\label{4.3}

Assume that the point in $\M_{\PP^2}(5, 1)$ represented by the morphism
\[
\f = \left[
\ba{ccc}
q & l & 0 \\
f_1 & q_1 & l_1 \\
f_2 & q_2 & l_2
\ea
\right]
\]
is fixed by $T$. Here $\f$ satisfies the conditions from the beginning of this section, namely:
$\det(\f) \neq 0$, $l \neq 0$, $l$ does not divide $q$, and $l_1$, $l_2$ are linearly independent.
For each $t \in T$, there is $(g(t), h(t)) \in G_2$ such that $t \f = h(t) \f g(t)$.
As at \cite[Section 2.2]{homology}, it is easy to see that $l_1$, $l_2$ are distinct monomials, and that $l$, $q$ are monomials.
Thus, we may write
\[
g(t) = \left[
\ba{ccc}
h_{11}^{-1} (t q)/q & 0 & 0 \\
0 & h_{11}^{-1} (t l)/l & 0 \\
u_1 & u_2 & 1
\ea
\right], \qquad h(t) = \left[
\ba{ccc}
h_{11} & 0 & 0 \\
v_1 & (t l_1)/l_1 & 0 \\
v_2 & 0 & (t l_2)/l_2
\ea
\right].
\]
Permuting, if necessary, rows two and three of $\f$, we may assume that $l$ is not a multiple of $l_1$.
Moreover, we may assume that $f_1$, $q_1$ do not contain any monomial divisible by $l_1$,
and that $q_1$, $q_2$ do not contain any monomial divisible by $l$.
From the relation
\[
t q_1^{} = v_1^{} h_{11}^{-1} (t l) + \frac{t l_1}{l_1} h_{11}^{-1} \frac{t l}{l} q_1^{} + u_2^{} (t l_1^{})
\]
we obtain the relations
\[
\tag{1}
t q_1^{} = \frac{t l_1}{l_1} h_{11}^{-1} \frac{t l}{l} q_1^{}, \quad v_1^{} = - a (t l_1^{}), \quad u_2^{} = a h_{11}^{-1} (t l)
\quad \text{for some $a \in \CC$}.
\]
From the relation
\[
t f_1^{} = v_1^{} h_{11}^{-1} (t q) + \frac{t l_1}{l_1} h_{11}^{-1} \frac{t q}{q} f_1^{} + u_1^{} (t l_1^{})
\]
we obtain the relations
\[
\tag{2}
t f_1^{} = \frac{t l_1}{l_1} h_{11}^{-1} \frac{t q}{q} f_1^{}, \qquad u_1^{} = a h_{11}^{-1} (t q).
\]
From the relation
\[
t q_2^{} = v_2^{} h_{11}^{-1} (t l) + \frac{t l_2}{l_2} h_{11}^{-1} \frac{t l}{l} q_2^{} + u_2^{} (t l_2^{})
\]
we obtain the relations
\[
\tag{3}
t q_2^{} = \frac{t l_2}{l_2} h_{11}^{-1} \frac{t l}{l} q_2^{}, \qquad v_2^{} = - a (t l_2^{}).
\]
Finally, we have the relation
\[
t f_2^{} = v_2^{} h_{11}^{-1} (t q) + \frac{t l_2}{l_2} h_{11}^{-1} \frac{t q}{q} f_2^{} + u_1^{} (t l_2^{}).
\]
Substituting the values for $u_1$ and $v_2$ found above yields the relation
\[
\tag{4}
t f_2^{} = \frac{t l_2}{l_2} h_{11}^{-1} \frac{t q}{q} f_2^{}.
\]
From relations (1)--(4) we deduce that $q_1$, $q_2$, $f_1$, $f_2$ are monomials.
In fact, $h_{11}^{-1}(t) = t_0^i t_1^j t_2^k$ for some integers $i$, $j$, $k$ satisfying the equation
$i + j + k = 0$ and
\[
\left[
\ba{cc}
f_1 & q_1 \\
f_2 & q_2
\ea
\right] = \left[
\ba{cc}
c_{11} X^i Y^j Z^k q l_1 & c_{12} X^i Y^j Z^k l l_1 \\
c_{21} X^i Y^j Z^k q l_2 & c_{22} X^i Y^j Z^k l l_2
\ea
\right],
\]
where $c_{rs} = 0$ if the corresponding monomial has negative exponents.

\noindent \\
Given $l_1$, $l_2$, $l$, $q$ as above and a monomial $d$ of degree $5$, we denote by $\M(l_1, l_2, l, q)$
the image in $\M_{\PP^2}(5,1)$ of the set of morphisms $\f \in W_2$ for which $\f_{11} = q$, $\f_{12} = l$,
and $\f_{23}$ has entries $l_1$, $l_2$. We denote by $\M(l_1, l_2, l, q, d)$ the subset given by the additional
condition $\det(\f) = d$.

\begin{proposition}
\label{4.3.1}
Assume that $l_1, l_2, l$ belong to $\{ X, Y, Z \}$, $l_1 \neq l_2$, $q$ belongs to $\{ X^2, Y^2, Z^2, XY, XZ, YZ \}$,
and $l$ does not divide $q$. Then, for any monomial $d$ of degree $5$ belonging to the ideal
$(l l_1, l l_2, q l_1, q l_2)$, the set of fixed points for the action of $T$ on $\M(l_1, l_2, l, q, d)$
has precisely one irreducible component, which is either a point or an affine line.
\end{proposition}

\begin{proof}
We will only examine the case when $l_1 = X$, $l_2 = Y$, $l = Y$, $q = XZ$, all other cases being analogous.
Consider, therefore, a morphism of the form
\[
\f = \left[
\ba{ccc}
XZ & Y & 0 \\
c_{11} X^{i+2} Y^j Z^{k+1} & c_{12} X^{i+1} Y^{j+1} Z^k & X \\
c_{21} X^{i+1} Y^{j+1} Z^{k+1} & c_{22} X^i Y^{j+2} Z^k & Y
\ea
\right],
\]
where $i + j + k = 0$. Assume, firstly, that $c_{12} \neq 0$.
Then $i = -1$ and $j = -1$ because of our assumption that $q_1$ be not divisible by $l$ or $l_1$.
Thus, $c_{11} = 0$, $c_{22} = 0$.
We obtain the fixed points
\[
\f_1(c) = \left[
\ba{ccc}
XZ & Y & 0 \\
0 & Z^2 & X \\
c Z^3 & 0 & Y
\ea
\right], \qquad c \in \CC \setminus \{ -1 \}.
\]
Assume now that $c_{12} = 0$ and $c_{11} \neq 0$. Then $i = -2$, hence $c_{21} = 0$, $c_{22} = 0$,
and we obtain the fixed points
\[
\left[
\ba{ccc}
XZ & Y & 0 \\
f & 0 & X \\
0 & 0 & Y
\ea
\right], \qquad f \in \{ Y^3, Y^2 Z, YZ^2, Z^3 \}.
\]
Assume next that $c_{11} = 0$, $c_{12} = 0$, $c_{22} \neq 0$. Then $j = -2$, hence $c_{21} = 0$,
and we obtain the fixed points
\[
\left[
\ba{ccc}
XZ & Y & 0 \\
0 & 0 & X \\
0 & q & Y
\ea
\right], \qquad q \in \{ X^2, XZ, Z^2 \}.
\]
In the final case to examine, when $c_{11} = 0$, $c_{12} = 0$, $c_{22} = 0$, we obtain the fixed points
\[
\left[
\ba{ccc}
XZ & Y & 0 \\
0 & 0 & X \\
f & 0 & Y
\ea
\right],
\]
where $f$ is any monomial of degree $3$.
For $f = Z^3$ we obtain a point in the moduli space, which we denote by $\f_1(\infty)$.
In conclusion, for the action of $T$ on $\M(X, Y, Y, XZ)$, we have sixteen fixed isolated points
and an affine line of fixed points, namely $\{ \f_1(c) \mid \ c \in \PP^1 \setminus \{ -1 \} \}$.
\end{proof}


\subsection{Fixed points in $\M_3$}
\label{4.4}

Clearly, $\O_Q(-P)(1)^\D$ gives a $T$-fixed point in $\M_{\PP^2}(5, 1)$ if and only if $Q$ is $T$-invariant
and $P$ is $T$-fixed. Thus, the torus fixed points in $\M_3$ are given by morphisms $\f \in W_3$
that are represented by matrices with entries monomials.


\section{The torus representation of the tangent spaces at the fixed points of $\M_{\PP^2}(5,1)$}
\label{section_5}

In the sequel, for each fixed point in $\M_k$ we will find a representative $\f \in W_k$
for which there exist diagonal matrices $u(t)$, $v(t)$ with entries characters $u_i(t)$, $v_j(t)$ of $(\CC^*)^3$,
such that $t \f = v(t) \f u(t)$ for all $t \in (\CC^*)^3$.
This allows us to apply the method of Section \ref{3.3}:
the action of $T$ on $\T_{\f} W_k$ is given by (\ref{3.3.1}),
the action of $T$ on $\T_{\f} (G_k \f)$ is given by (\ref{3.3.2}) and, as $T$-modules, the quotient
$\T_{\f} W_k / \T_{\f} (G_k \f)$ is isomorphic to $\T_{[\f]} \M_{\PP^2}(5, 1)$.

\subsection*{Notations}

\noindent \\ \\
\begin{tabular}{r c l}
$\Sym^l$ & = & $\{ X^i Y^j Z^k \mid \ i,j,k \in \ZZ, \ i ,j, k \ge 0,\ i+j+k = l \}$, where $l \ge 1$; \\
$\{ x, y, z \}$ & = & the standard basis for the lattice of characters of $(\CC^*)^3$; \\
$\chi_0$ & = & the trivial character of $T$; \\
$\s^l$ & = & $\{ ix + jy + kz \mid \ i,j,k \in \ZZ, \ i,j,k \ge 0,\ i + j + k = l \}$, where $l \ge 1$.
\end{tabular}

\noindent \\
In this section we will use additive notation to denote characters of $T$ or of $(\CC^*)^3$. Thus,
\[
\chi^*(T) = \{ ix + jy + kz \mid \ i, j, k \in \ZZ,\ i + j + k = 0 \}.
\]
We also adopt the following convention: whenever a monomial $X^i Y^j Z^k$ appears in a list of characters,
it stands for the expression $ix + jy + kz$.


\subsection{Fixed points in $\M_0$}
\label{5.1}

The action of $T$ on $\T_{\f} W_0$ is given by (\ref{3.3.1}).
Let $i$, $j$, $k$ be non-negative integers.
Let $w_{mn}^{ijk}$ denote the matrix having entry $X^i Y^j Z^k$ on position $(m, n)$
and entries zero everywhere else. Viewed as a tangent vector, $w_{mn}^{ijk}$ is acted by $T$
with weight $-v_m - u_n + ix + jy + kz$. Now $\T_{\f} W_0$ is the space of $4 \times 4$-matrices with
entries linear forms on the first three rows and quadratic forms on the fourth row.
It follows that the set
\[
\{ w_{mn}^{ijk} \mid m = 1, 2, 3, \, n = 1, 2, 3, 4, \, i+j+k = 1 \} \cup
\{ w_{4n}^{ijk} \mid n = 1, 2, 3, 4, \, i+j+k = 2 \}
\]
forms a basis of the tangent space. 
The corresponding list of weights is organised in the following array:
\[
\ba{rrrrrrr}
- v_1 - u_1 + \s^1 & \quad & - v_1 - u_2 + \s^1 & \quad & - v_1 - u_3 + \s^1 & \quad & - v_1 - u_4 + \s^1 \\
- v_2 - u_1 + \s^1 & \quad & - v_2 - u_2 + \s^1 & \quad & - v_2 - u_3 + \s^1 & \quad & - v_2 - u_4 + \s^1 \\
- v_3 - u_1 + \s^1 & \quad & - v_3 - u_2 + \s^1 & \quad & - v_3 - u_3 + \s^1 & \quad & - v_3 - u_4 + \s^1 \\
- v_4 - u_1 + \s^2 & \quad & - v_4 - u_2 + \s^2 & \quad & - v_4 - u_3 + \s^2 & \quad & - v_4 - u_4 + \s^2
\ea
\]
The group acting by conjugation on $W_0$ is 
\[
G_0 = (\Aut(4\O(-2)) \times \Aut(3\O(-1) \oplus \O))/\CC^*.
\]
Let $e$ denote its neutral element. We identify canonically $\T_{\f} (G_0 \f)$ with the space
\[
\T_e G_0 = (\End(4\O(-2)) \oplus \End(3\O(-1) \oplus \O))/\CC,
\]
which is represented by pairs $(A, B)$ of matrices with polynomial entries.
The $T$-action on this tangent space is given by (\ref{3.3.2}).
Let $A_{mn}^{ijk}$ and $B_{mn}^{ijk}$ be the matrices defined in the same way as $w_{mn}^{ijk}$.
Now $T$ acts on the tangent vector $(A_{mn}^{ijk}, 0)$ with weight $u_m - u_n + ix + jy + kz$,
and on $(0, B_{mn}^{ijk})$ with weight $-v_m + v_n + ix + jy + kz$.
The subspace of diagonal matrices $\{ (c I, c I)$, $c \in \CC \}$ that we quotient out is acted upon trivially.
A basis for $\End(4\O(-2))$ is given by
\[
\{ A_{mn}^{000} \mid 1 \le m , n \le 4 \}.
\]
Also,
\[
\{ B_{mn}^{000} \mid 1 \le m, n \le 3 \} \cup \{ B_{4n}^{ijk} \mid 1 \le n \le 3, \, i+j+k = 1 \} \cup \{ B_{44}^{000} \}
\]
forms a basis of $\End(3\O(-1) \oplus \O)$.
To get the list of weights for $\T_e G_0$ we add the two lists for the two bases above and we subtract $\{ \chi_0 \}$.
The result is expressed in the following tableau:
\[
\ba{lllllll}
\phantom{-} \chi_0 & \quad & \phantom{-} u_1 - u_2 & \quad & \phantom{-} u_1 - u_3 & \quad & u_1 - u_4 \\
\phantom{-} u_2 - u_1 & \quad & \phantom{-} \chi_0 & \quad & \phantom{-} u_2 - u_3 & \quad & u_2 - u_4 \\
\phantom{-} u_3 - u_1 & \quad & \phantom{-} u_3 - u_2 & \quad & \phantom{-} \chi_0 & \quad & u_3 - u_4 \\
\phantom{-} u_4 - u_1 & \quad & \phantom{-} u_4 - u_2 & \quad & \phantom{-} u_4 - u_3 & \quad & \chi_0 \\
\phantom{-} \chi_0 & \quad & - v_1 + v_2 & \quad & - v_1 + v_3 \\
- v_2 + v_1 & \quad & \phantom{-} \chi_0 & \quad & - v_2 + v_3 \\
- v_3 + v_1 & \quad & - v_3 + v_2 & \quad & \phantom{-} \chi_0 \\
- v_4 + v_1 + \s^1 & \quad & - v_4 + v_2 + \s^1 & \quad & - v_4 + v_3 + \s^1
\ea
\]
According to Section \ref{4.1}, the fixed points in $\M_0$ are represented by matrices $\f$,
where, modulo a permutation of variables, $\f_{11}$ is one among the matrices
$\alpha$, $\beta$, $\gamma$, $\delta$, $\varepsilon$, $\zeta$, $\eta$,
$\theta$, $\iota$, $\kappa$, $\lambda(a : b)$, $\mu$, $\nu$, and $\xi$ from Section \ref{section_2},
Let $\zeta_i$, $1 \le i \le 4$, be the maximal minor of $\f_{11}$ obtained by deleting column $i$.
If $\f_{11} \neq \lambda(1 : -1)$ and $\f_{11} \neq \xi$, then all $\zeta_i$ are non-zero,
and we may assume that $\f$ has the form
\[
{\mathversion{bold}
\boldsymbol{\f_{11}(d)} = \left[
\ba{c}
\f_{11} \\
\ba{cccc}
c_1 d/\zeta_1 & c_2 d/\zeta_2 & c_3 d/\zeta_3 & c_4 d/\zeta_4
\ea
\ea
\right],}
\]
where $d$ is a monomial of degree $5$ belonging to the ideal $(\zeta_1, \zeta_2, \zeta_3, \zeta_4)$
and $c_i = 0$ if $d$ is not divisible by $\zeta_i$.
If $\f_{11} = \lambda(1 : -1)$ or $\f_{11} = \xi$, then $\zeta_4 = 0$ and the other $\zeta_i$ are non-zero.
However, the torus representation of the tangent space at a point does not change when the point
varies in a connected component of $\M_{\PP^2}(5,1)^T$.
For this reason, we may ignore the case when $\f_{11} = \lambda(1 : -1)$,
in fact, we may restrict to the case when $\f_{11} = \lambda$.
Without proof, we claim that, for each feasible $d$, there is a point in $\M(\xi, d)$ (notation as at Section \ref{4.1})
represented by a matrix of the form
\[
{\mathversion{bold}
\boldsymbol{\xi(d)} = \left[
\ba{c}
\xi \\
\ba{cccc}
c_1 d/\zeta_1 & c_2 d/\zeta_2 & c_3 d/\zeta_3 & 0
\ea
\ea
\right].}
\]
In conclusion, we will only consider matrices of the form $\psi(d)$, where
\[
\psi \in \{ \alpha, \, \beta, \, \gamma, \, \delta, \, \varepsilon, \, \zeta, \, \eta, \, \theta, \, \iota, \, \kappa, \, \lambda, \, \mu, \, \nu, \, \xi \}.
\]
The values of $v_1$, $v_2$, $v_3$, $u_1$, $u_2$, $u_3$, $u_4$ for $\psi(d)$ are the same as the values for $\psi$,
as given in Table 1, Section \ref{3.3}.
On a case-by-case basis we can show that the equation
\[
d - \zeta_1 - u_1 = d - \zeta_2 - u_2 = d - \zeta_3 - u_3 = d - \zeta_4 - u_4
\]
of characters (there being no last term if $\psi = \xi$) holds.
This is the character $v_4$ of $\psi(d)$.
Its values are given in Table 4 below.

\begin{table}[!hpt]{Table 4}
\begin{center}
\begin{tabular}{| c | l | c | c | l |}
\cline{1-2} \cline{4-5}
Fixed point & $v_4$
& &
Fixed point & $v_4$
\\
\cline{1-2} \cline{4-5}
$\alpha(d)$ & $d - x - y - z$
& &
$\beta(d)$ & $d - x - y - z$
\\
\cline{1-2} \cline{4-5}
$\gamma(d)$ & $d - x - y - z$
& &
$\delta(d)$ & $d - 2x - z$
\\
\cline{1-2} \cline{4-5}
$\varepsilon(d)$ & $d - 2x - y$
& &
$\zeta(d)$ & $d - 2x - y$
\\
\cline{1-2} \cline{4-5}
$\eta(d)$ & $d - x - 2y$
& &
$\theta(d)$ & $d - x - y - z$
\\
\cline{1-2} \cline{4-5}
$\iota(d)$ & $d - x - 2z$
& &
$\kappa(d)$ & $d - x - 2y$
\\
\cline{1-2} \cline{4-5}
$\lambda(d)$ & $d - 3x$
& &
$\mu(d)$ & $d - 2x - y$
\\
\cline{1-2} \cline{4-5}
$\nu(d)$ & $d - x - 2y$
& &
$\xi(d)$ & $d - 2x - y$
\\
\cline{1-2} \cline{4-5}
\end{tabular}
\end{center}
\end{table}


\subsection{Fixed points in $\M_1$}
\label{5.2}

By analogy with Section \ref{5.1},
the list of weights for the action of $T$ on $\T_{\f} W_1$ is expressed by the table
\[
\ba{rrr}
- v_1 - u_1 + \s^3 & \quad & - v_1 - u_2 + \s^2 \\
- v_2 - u_1 + \s^3 & \quad & - v_2 - u_2 + \s^2
\ea
\]
and the list of weights for the action of $T$ on $\T_{\f} (G_1 \f)$ is represented by the array
\[
\ba{lllll}
\chi_0 & \quad & \phantom{-} \chi_0 & \quad & - v_1 + v_2 \\
u_2 - u_1 + \s^1 & \quad & - v_2 + v_1 & \quad & \phantom{-} \chi_0
\ea
\]
Let $\F = \Coker(\f)$.
By analogy with \cite[Proposition 6.2]{choi_maican}, we have the following description of the
torus action on the normal space at $[\F]$.

\begin{proposition}
\label{5.2.1}
The normal space $N_{[\F]}$ to $\M_1$ at $[\F]$ can be identified with
\[
\H^0(\F)^* \tensor \H^1(\F).
\]
The torus acts on $N_{[\F]}$ with weights $u_1 + v_1 - x - y - z$ and $u_1 + v_2 - x - y - z$.
\end{proposition}

\noindent
Recall from Proposition \ref{4.2.1} that an irreducible component of $\M_1^T$ is uniquely determined
by $q_1$, $q_2$ and $d$.
Thus, $\f$ has the form
\[
{\mathversion{bold}
\boldsymbol{\omicron(q_1, q_2, d)} = \left[
\ba{cc}
c_1 d/q_2 & q_1 \\
c_2 d/q_1 & q_2
\ea
\right].}
\]
The characters $u_1$, $u_2$, $v_1$, $v_2$ have to be chosen such that (using additive notation)
\[
\ba{lll}
v_1 + u_1 = d - q_2, & \ & v_1 + u_2 = q_1, \\
v_2 + u_1 = d - q_1, & \ & v_2 + u_2 = q_2.
\ea
\]
Clearly, we may choose
\[
u_1 = d - q_1 - q_2, \qquad u_2 = 0, \qquad v_1 = q_1, \qquad v_2 = q_2.
\]
With the aid of Proposition \ref{5.2.1} we can determine which points in $\M_1^T$ are accumulation points
of lines or surfaces in $\M_0^T$.
Indeed, $[\F]$ lies in the closure of a positive-dimensional component of $\M_0^T$
if and only if at least one among the weights
\begin{align*}
u_1 + v_1 - x - y - z & = d - q_2 - x - y - z, \\
u_1 + v_2 - x - y - z & = d - q_1 - x - y - z
\end{align*}
equals $\chi_0$.
As $q_1 \neq q_2$, both weights cannot equal $\chi_0$ at the same time.
Thus, there are two possible situations.
If $[\F]$ is an isolated point for the action of $T$ on $\M_1$ and one among the above weights is $\chi_0$,
then $[\F]$ is the limit point of a line in $\M_0^T$.
If $[\F]$ belongs to a line of $\M_1^T$ and one of the above weights is $\chi_0$,
then $[\F]$ is the limit point of a surface in $\M_0^T$.
The first column of Table 5 below lists the pairs $(q_1, q_2)$ up to a permutation of variables.
The second column contains the monomials $d = X^i Y^j Z^k$ of degree $5$ that are in the ideal generated by $q_1$
and $q_2$. The third column contains the values of $d$, found in Section \ref{4.2}, for which $\M(q_1, q_2, d)^T$ is a line.
The fourth column lists the values of $d$ for which $\omicron(q_1, q_2, d)$ is the limit point of a line in $\M_{0}^T$.
The last column lists the values of $d$ for which $\omicron(q_1, q_2, d)$ is the limit point of a surface in $\M_{0}^T$.

\begin{table}[!hpt]{Table 5. Fixed points $\omicron(q_1, q_2, d)$ in $\M_1$.}
\begin{center}
\begin{tabular}{| l | l | l | l | l |}
\hline
$(q_1, q_2)$ &
\begin{tabular}{l} $d$ \end{tabular} &
\begin{tabular}{l} Affine \\ lines \end{tabular} &
\begin{tabular}{l} Limit \\ points \\ of lines \end{tabular} &
\begin{tabular}{l} Limit \\ points \\ of surfaces \end{tabular}
\\
\hline \hline
$(X^2, Y^2)$ &
\begin{tabular}{l} $\Sym^5 \setminus \{ Z^5, Z^4 X, Z^4 Y, Z^3 XY \}$ \end{tabular} & &
\begin{tabular}{l} $\vphantom{\overline{X^X}} X^3 YZ$ \\ $XY^3 Z\phantom{^2}$ \end{tabular} &
\\
\hline
$(X^2, YZ)$ &
\begin{tabular}{l} $\Sym^5 \setminus \{ Y^5, XY^4, Z^5, XZ^4 \}$ \end{tabular} & &
\begin{tabular}{l} $\vphantom{\overline{X^X}} X^3 YZ$ \\ $XY^2 Z^2$ \end{tabular} &
\\
\hline
$(X^2, XY)$ &
\begin{tabular}{l} $\Sym^5 \setminus \{ Z^5, XZ^4, YZ^4, Y^2 Z^3,$ \\ $\phantom{\Sym^5 \setminus} Y^3 Z^2, Y^4 Z, Y^5 \}$ \end{tabular} &
\begin{tabular}{l} $\vphantom{\overline{X^X}} X^2 Y^3$ \\ $X^2 YZ^2$ \\ $X^2 Y^2 Z$ \end{tabular} &
\begin{tabular}{l} $X^3 YZ$ \end{tabular} &
\begin{tabular}{l} $X^2 Y^2 Z$ \end{tabular}
\\
\hline
$(XZ, YZ)$ &
\begin{tabular}{l} $\Sym^5 \setminus \{ X^5, X^4 Y, X^3 Y^2,$ \\ $\phantom{\Sym^5 \setminus} X^2 Y^3, XY^4, Y^5, Z^5 \}$ \end{tabular} &
\begin{tabular}{l} $\vphantom{\overline{X^X}} X^3 YZ$ \\ $XY^3 Z$ \\ $X^2 Y^2 Z$ \end{tabular} &
\begin{tabular}{l} $X^2 YZ^2$ \\ $XY^2 Z^2$ \end{tabular} &
\\
\hline
\end{tabular}
\end{center}
\end{table}


\subsection{Fixed points in $\M_2$}
\label{5.3}

Recall from Proposition \ref{4.3.1} that an irreducible component of $\M_2^T$ is uniquely determined
by $l_1$, $l_2$, $l$, $q$ and $d$. Thus, $\f$ has the form
\[
{\mathversion{bold}
\boldsymbol{\pi(l_1, l_2, l, q, d)} = \left[
\ba{ccc}
q & l & 0 \\
c_{11} d/l l_2 & c_{12} d/q l_2 & l_1 \\
c_{21} d/l l_1 & c_{22} d/q l_1 & l_2
\ea
\right].}
\]
The characters $u_1$, $u_2$, $u_3$, $v_1$, $v_2$, $v_3$ have to be chosen such that
(adopting additive notation)
\[
\ba{lllll}
v_1 + u_1 = q, & \ & v_1 + u_2 = l, \\
v_2 + u_1 = d - l - l_2, & \ & v_2 + u_2 = d - q - l_2, & \ & v_2 + u_3 = l_1, \\
v_3 + u_1 = d - l - l_1, & \ & v_3 + u_2 = d - q - l_1, & \ & v_3 + u_3 = l_2.
\ea
\]
Clearly, we may choose
\[
\ba{lll}
u_1 = d - l - l_1 - l_2 & \qquad & v_1 = - d + q + l + l_1 + l_2 \\
u_2 = d - q - l_1 - l_2 & \qquad & v_2 = l_1 \\
u_3 = 0 & \qquad & v_3 = l_2
\ea
\]
By analogy with Section \ref{5.1}, the list of weights for the action of $T$ on $\T_{\f} W_2$ is represented by the tableau
\[
\ba{rrrrr}
- v_1 - u_1 + \s^2 & \quad & - v_1 - u_2 + \s^1 \\
- v_2 - u_1 + \s^3 & \quad & - v_2 - u_2 + \s^2 & \quad & - v_2 - u_3 + \s^1 \\
- v_3 - u_1 + \s^3 & \quad & - v_3 - u_2 + \s^2 & \quad & - v_3 - u_3 + \s^1
\ea
\]
Observe that $\f$ has a stabiliser of dimension one consisting of matrices of the form
\[
\left( \left[
\ba{ccc}
1 & 0 & 0 \\
0 & 1 & 0 \\
cq & cl & 1
\ea
\right], \ \left[
\ba{ccc}
1 & 0 & 0 \\
c l_1 & 1 & 0 \\
c l_2 & 0 & 1
\ea
\right] \right),
\]
where $c \in \CC$. Thus, $\T_{e} \Stab(\f)$ is spanned by the tangent vector
\[
s = \left( \left[
\ba{ccc}
0 & 0 & 0 \\
0 & 0 & 0 \\
q & l & 0
\ea
\right], \ \left[
\ba{ccc}
0 & 0 & 0 \\
l_1 & 0 & 0 \\
l_2 & 0 & 0
\ea
\right] \right).
\]
Let $s_1$ be obtained by setting $l = l_1 = l_2 = 0$ in the above expression,
let $s_2$ be obtained by setting $q = l_1 = l_2 = 0$,
let $s_3$ be obtained by setting $q = l = l_2 = 0$,
and let $s_4$ be obtained by setting $q = l = l_1 = 0$.
The torus acts on $s_1$, $s_2$, $s_3$, $s_4$ with weights
\[
u_3 - u_1 + q, \qquad u_3 - u_2 + l, \qquad v_1 - v_2 + l_1, \qquad v_1 - v_3 + l_2.
\]
Substituting the values for $u_1$, $u_2$, $v_1$, $v_2$ from above we get the same expression
in all four cases, namely
\[
- d + q + l + l_1 + l_2.
\]
Since $s=s_1 + s_2 + s_3 + s_4$, this is the weight for the action of $T$ on $\T_{e} \Stab(\f)$.
To get the list of weights for the action of $T$ on $\T_{\f} (G_2 \f)$ we need to subtract this weight from
the list
\[
\ba{lllll}
\phantom{-} \chi_0 \\
\phantom{-} u_2 - u_1 + \s^1 & \quad & \phantom{-} \chi_0 \\
\phantom{-} u_3 - u_1 + \s^2 & \quad & \phantom{-} u_3 - u_2 + \s^1 & \quad & \phantom{-} \chi_0 \\
- v_2 + v_1 + \s^1 & \quad & \phantom{-} \chi_0 & \quad & - v_2 + v_3 \\
- v_3 + v_1 + \s^1 & \quad & - v_3 + v_2 & \quad & \phantom{-} \chi_0
\ea
\]

\begin{proposition}
\label{5.3.1}
Let $\F = \Coker(\f)$. Let $N_{[\F]}$ be the normal space to $\M_2$ at $[\F]$.
The torus $T$ acts on $N_{[\F]}$ with weights
\[
u_1 + v_2 - x - y - z, \qquad u_1 + v_3 - x - y - z, \qquad - v_1 - u_3.
\]
\end{proposition}

\begin{proof}
By analogy with \cite[Theorem 4.3.3]{dedicata},
we can show that $\M_1 \cup \M_2$ is a locally closed smooth subvariety of $\M_{\PP^2}(5,1)$ of codimension $2$
that is the geometric quotient of an open subset $\WW \subset \WW_2$ modulo $G_2$.
Moreover, $\M_2$ has codimension $1$ in $\M_1 \cup \M_2$.
In fact, $\WW$ is the subset of injective morphisms that have semi-stable cokernel.
The table of weights for the action of $T$ on $\T_{\f} \WW$ is the same as the table for $\T_{\f} W_2$ except that it contains
the weight $- v_1 - u_3$ in the upper-right corner, accounting for the normal direction to $\M_2$ inside $\M_1 \cup \M_2$.
The other two weights account for the two normal directions to $\M_1 \cup \M_2$, as in Proposition \ref{5.2.1}.
\end{proof}

\noindent
In view of the above proposition, $[\F]$ lies in the closure of a positive-dimensional component of $(\M_0 \cup \M_1)^T$
if and only if at least one among the weights
\begin{align*}
u_1 + v_2 - x - y - z & = d - l - l_2 - x - y - z, \\
u_1 + v_3 - x - y - z & = d - l - l_1 - x - y - z, \\
- v_1 - u_3 & = d - q - l - l_1 - l_2
\end{align*}
equals $\chi_0$. In Table 6 below, which is organised as Table 5, we have the information regarding the fixed points in $\M_2$.
We assume that $l_1 = X$, $l_2 = Y$, the other cases being obtained by a permutation of variables.

\begin{table}[!hpt]{Table 6. Fixed points $\pi(X, Y, l, q, d)$ in $\M_2$.}
\begin{center}
\begin{tabular}{| l | l | l | l | l |}
\hline
$(l, q)$ & $d$ &
\begin{tabular}{l} Affine \\ lines \end{tabular} &
\begin{tabular}{l} Limit \\ points \\ of lines \end{tabular} &
\begin{tabular}{l} Limit \\ points \\ of surfaces \end{tabular}
\\
\hline \hline
$(Y, Z^2)$ & $\Sym^5 \setminus \{ X^5, Z^5, X^4 Z \}$ & &
\begin{tabular}{l} $\vphantom{\overline{X^X}} X^2 Y^2 Z$ \\ $XY^3 Z$ \\ $XY^2 Z^2$ \end{tabular} &
\\
\hline
$(Y, X^2)$ & $\Sym^5 \setminus \{ Z^5, XZ^4, X^2 Z^3, YZ^4 \}$ & $X^2 YZ^2$ &
\begin{tabular}{l} $\vphantom{\overline{X^X}} X^2 Y^2 Z$ \\ $XY^3 Z$ \\ $X^3 Y^2$ \end{tabular} &
\\
\hline
$(Y, XZ)$ & $\Sym^5 \setminus \{ X^5, Z^5, XZ^4, YZ^4 \} \vphantom{\overline{\Sym^5}}$ & $XYZ^3$ &
\begin{tabular}{l} $XY^3 Z\phantom{^2}$ \end{tabular} &
\begin{tabular}{l} $X^2 Y^2 Z$ \end{tabular}
\\
\hline
$(Z, X^2)$ & $\Sym^5 \setminus \{ Y^5, Z^5, XY^4 \}$ & &
\begin{tabular}{l} $\vphantom{\overline{X^X}} X^2 YZ^2$ \\ $XY^2 Z^2$ \\ $X^3 YZ$ \end{tabular} & \\
\hline
$(Z, XY)$ & $\Sym^5 \setminus \{ X^5, Y^5, Z^5 \}$ & &
\begin{tabular}{l} $\vphantom{\overline{X^X}} X^2 YZ^2$ \\ $XY^2 Z^2$ \\ $X^2 Y^2 Z$ \end{tabular} & \\
\hline
\end{tabular}
\end{center}
\end{table}


\subsection{Fixed points in $\M_3$}
\label{5.4}

By analogy with Section \ref{5.1}, the list of weights for the action of $T$ on $\T_{\f} W_3$ reads
\[
\ba{rrr}
- v_1 - u_1 + \s^1 & \quad & - v_1 - u_2 + \s^1 \\
- v_2 - u_1 + \s^4 & \quad & - v_2 - u_2 + \s^4
\ea
\]
and the list of weights for the action of $T$ on $\T_{\f}(G_3 \f)$ is expressed in the tableau
\[
\ba{lllll}
\chi_0 & \quad & u_1 - u_2 & \quad & \phantom{-} \chi_0 \\
u_2 - u_1 & \quad & \chi_0 & \quad & - v_2 + v_1 + \s^3
\ea
\]
Recall from Section \ref{4.4} that the action of $T$ on $\M_3$ has only isolated fixed points,
given by morphisms of the form
\[
{\mathversion{bold}
\boldsymbol{\rho(l_1, l_2, d)} = \left[
\ba{cc}
l_1 & l_2 \\
c_1 d/l_2 & c_2 d/l_1
\ea
\right].}
\]
The characters $u_1$, $u_2$, $v_1$, $v_2$ have to be chosen such that (using additive notation)
\[
\ba{lll}
v_1 + u_1 = l_1, & \ & v_1 + u_2 = l_2, \\
v_2 + u_1 = d - l_2, & \ & v_2 + u_2 = d - l_1.
\ea
\]
Clearly, we may choose $u_1 = l_1$, $u_2 = l_2$, $v_1 = 0$, $v_2 = d - l_1 - l_2$.

\begin{proposition}
\label{5.4.1}
Let $\F = \Coker(\f)$. Let $N_{[\F]}$ be the normal space to $\M_3$ at $[\F]$.
Then we have a canonical isomorphism
\[
N_{[\F]} \isom \H^0(\F(-1))^* \tensor \H^1(\F(-1)).
\]
Denote $\{ l_3 \} = \{ X, Y, Z \} \setminus \{ l_1, l_2 \}$. Then the torus $T$ acts on $N_{[\F]}$ with weights
\begin{align*}
& d - \phantom{2} l_1 - 3 l_2 - \phantom{2} l_3, \\
& d - 2 l_1 - 2 l_2 - \phantom{2} l_3, \\
& d - 3 l_1 - \phantom{2} l_2 - \phantom{2} l_3, \\
& d - \phantom{2} l_1 - 2 l_2 - 2 l_3, \\
& d - 2 l_1 - \phantom{2} l_2 - 2 l_3.
\end{align*}
\end{proposition}

\begin{proof}
Denote $\G = \F^\D(1)$.
According to \cite[Lemma 3]{rendiconti}, dualising the resolution for $\F$ yields the resolution
\[
0 \lra \O(-3) \oplus \O \stackrel{\psi}{\lra} 2\O(1) \lra \G \lra 0,
\]
\[
\psi = \left[
\ba{cc}
c_1 d/l_2 & l_1 \\
c_2 d/l_1 & l_2
\ea
\right].
\]
This allows us to use the argument at  \cite[Proposition 6.2]{choi_maican}.
Applying the $\Ext(\_, \G)$ functor to the canonical morphism $\H^0(\G) \tensor \O \to \G$
yields the linear map
\[
\epsilon \colon \Ext^1(\G, \G) \lra \H^0(\G)^* \tensor \H^1(\G).
\]
Its kernel is a subspace of $\T_{[\G]} \M_3^\D \subset \T_{[\G]} \M_{\PP^2} (5,4)$.
Recall that $\T_{[\G]} \M_3^\D$ has codimension $5$ in $\Ext^1(\G, \G)$.
Since $\dim (\H^0(\G)^* \tensor \H^1(\G)) = 5$, we deduce that $\epsilon$ is surjective
and that $\operatorname{Ker}(\epsilon) = \T_{[\G]} \M_3^\D$.
Thus, we obtain a canonical isomorphism
\[
N_{[\F^\D(1)]} \isom \H^0(\F^\D(1))^* \tensor \H^1(\F^\D(1)).
\]
Using Serre Duality, as at \cite[Proposition 3.3.1]{homology}, yields the canonical isomorphism from the proposition.
We have identifications
\[
\H^0(\G) = \H^0(2\O(1))/\H^0(\O) = (V^* \oplus V^*) / \CC(l_1, l_2).
\]
The vectors $(X, 0)$, $(Y, 0)$, $(Z, 0)$, $(0, X)$, $(0, Y)$, $(0, Z)$ form a basis of $V^* \oplus V^*$.
The calculations at \cite[Proposition 6.2]{choi_maican} show that these are eigenvectors for the action of
$(\CC^*)^3$, corresponding to the weights
\[
- u_1 + x, \quad - u_1 + y, \quad - u_1 + z, \quad - u_2 + x, \quad - u_2 + y, \quad - u_2 + z.
\]
The vector $(l_1, l_2)$ is acted on trivially. Moreover, $(\CC^*)^3$ acts on $\H^1(\G)$ with weight
$v_2 - x - y - z$. It follows that the list of weights for the action of $T$ on $N_{[\G]}$ is obtained by subtracting
the weight
\[
v_2 - x - y - z = d - 2 l_1 - 2 l_2 - l_3
\]
from the list
\begin{align*}
u_1 & + v_2 - 2 x - \phantom{2} y - \phantom{2} z, \\
u_1 & + v_2 - \phantom{2} x - 2y - \phantom{2} z, \\
u_1 & + v_2 - \phantom{2} x - \phantom{2} y - 2z, \\
u_2 & + v_2 - 2 x - \phantom{2} y - \phantom{2} z, \\
u_2 & + v_2 - \phantom{2} x - 2y - \phantom{2} z, \\
u_2 & + v_2 - \phantom{2} x - \phantom{2} y - 2z,
\end{align*}
which is the same as the list
\begin{align*}
& d - 2 l_1 - 2 l_2 - \phantom{2} l_3, \\
& d - \phantom{2} l_1 - 3 l_2 - \phantom{2} l_3, \\
& d - \phantom{2} l_1 - 2 l_2 - 2 l_3, \\
& d - 3 l_1 - \phantom{2} l_2 - \phantom{2} l_3, \\
& d - 2 l_1 - 2 l_2 - \phantom{2} l_3, \\
& d - 2 l_1 - \phantom{2} l_2 - 2 l_3.
\end{align*}
In view of the fact that $N_{[\G]}$ and $N_{[\F]}$ are isomorphic as $T$-modules, this proves the proposition.
\end{proof}

\noindent
The weights for the action of $T$ on $N_{[\F]}$ are distinct, so at most one of them can be $\chi_0$.
This shows that no point of $\M_3^T$ lies in a two-dimensional component of $\M_{\PP^2}(5, 1)^T$.
The points of $\M_3^T$ lying on projective lines inside $\M_{\PP^2}(5, 1)^T$ are listed in the third column of Table 7 below.

\begin{table}[!hpt]{Table 7. Fixed points $\rho(l_1, l_2, d)$ in $\M_3$.}
\begin{center}
\begin{tabular}{| l | l | l |}
\hline
$(l_1, l_2)$ & $d$ & Limit points of lines
\\
\hline \hline
$(X, Y)$ & $\Sym^5 \setminus \{ Z^5 \}$ & $\vphantom{\overline{X^X}}XY^3 Z$, $X^2 Y^2 Z$, $X^3 YZ$, $XY^2 Z^2$, $X^2 YZ^2$
\\
\hline
$(X, Z)$ & $\Sym^5 \setminus \{ Y^5 \}$ & $\vphantom{\overline{X^X}}XYZ^3$, $X^2 YZ^2$, $X^3 YZ$, $XY^2 Z^2$, $X^2 Y^2 Z$
\\
\hline
$(Y, Z)$ & $\Sym^5 \setminus \{ X^5 \}$ & $\vphantom{\overline{X^X}}XYZ^3$, $XY^2 Z^2$, $XY^3 Z$, $X^2 YZ^2$, $X^2Y^2 Z$
\\
\hline
\end{tabular}
\end{center}
\end{table}


\subsection{Proof of Theorem 2}
\label{5.5}
The following lemma is probably well-known, but we need it in order to determine the structure of the irreducible components
of dimension $2$ of the torus fixed locus.

\begin{lemma}
\label{5.5.1}
Let $\E$ be a vector bundle of rank $2$ on $\PP^n$. Let $s$ be a section of $\PP(\E)$.
Assume that $\PP(\E) \setminus \{ s \}$ is the trivial bundle on $\PP^n$ with fibre $\AA^1$.
Then $\PP(\E)$ is the trivial bundle on $\PP^n$ with fibre $\PP^1$.
\end{lemma}

\begin{proof}
Tensoring, possibly, $\E$ with a line bundle, we may assume that $s$ lifts to a global section $\tilde{s}$ of $\E$.
The map $\O \to \E$ of multiplication with $\tilde{s}$ is injective and its cokernel is a line bundle $\L$.
The total space of $\L^*$ is isomorphic to $\PP(\E) \setminus \{ s \}$, hence $\L^*$ is trivial.
Clearly, $\E \isom \O \oplus \L$, hence $\E$ is trivial, hence $\PP(\E)$ is trivial.
\end{proof}

\begin{proposition}
\label{5.5.2}
Each irreducible component of dimension $2$ of $\M_{\PP^2}(5, 1)^T$ is isomorphic to $\PP^1 \times \PP^1$.
\end{proposition}

\begin{proof}
Recall the surface $\Sigma_0$ from Section \ref{4.1} and the line $\Lambda$ from Section \ref{2.2.7}.
Denote $\Sigma = \overline{\Sigma}_0$.
An examination of Tables 5, 6 and 7 convinces us that
\[
\bigcup_{\sigma \in \Sym_3} \Sigma_{\sigma} \cap ( \M_1 \cup \M_2 \cup \M_3 )
= \bigcup_{\sigma \in \Sym_3} (\omicron(X^2, XY, X^2 Y^2 Z)_{\sigma} \cup \pi(X, Y, Y, XZ, X^2 Y^2 Z)_{\sigma}).
\]
It follows that $s = \Sigma \setminus \Sigma_0$ is a section for the map $\Sigma \to \Lambda$.
Thus, $\Sigma = \PP(\E)$ for a vector bundle $\E$ of rank $2$ over the projective line $\Lambda$.
The proposition follows from Lemma \ref{5.5.1}, in view of the fact that $\Sigma_0$ is isomorphic to the trivial bundle over $\Lambda$
with fibre $\AA^1$.
\end{proof}

\noindent
Denote $\M = \M_{\PP^2}(5, 1)$.
The first part of Theorem \ref{theorem_2} concerning the structure of $\M^T$ follows from Section \ref{section_4}
and Proposition \ref{5.5.2}.

Let $\lambda(\tau) = (\tau^{n_0}, \tau^{n_1}, \tau^{n_2})$ denote a one-parameter subgroup of $T$
that is not orthogonal to any non-zero character $\chi$ for which there is $[\f ] \in \M^T$ such that the eigenspace
\[
(\T_{[\f]} \M)_{\chi} = \{ w \in \T_{[\f]} \M \mid \ t w = \chi(t) w \text{ for all } t \in T \}
\]
is non-zero. 
From the results in this section it follows that the set of such characters $\chi$ is contained in the set
\[
\{ ix+ jy + kz \mid -7 \le i, j, k \le 7 \}.
\]
Thus, we can choose $\lambda(\tau) = (1, \tau, \tau^8)$.
Recall from Section \ref{3.3} that for a point $[\f] \in \M^T$ the integer $p [\f]$ is the sum of the dimensions
of the spaces $(\T_{[\f]} \M)_{\chi}$ for which $\langle \lambda, \chi \rangle > 0$.
If $[\f]$ varies in an irreducible component $X$ of $\M^T$, then $p [\f]$ does not change, so we may define the integer $p(X)$.
These integers can be computed with the help of the {\sc Singular} \cite{singular} program from Appendix \ref{appendix_B}.
From (\ref{3.3.3}) we deduce the formula
\[
P_{\M}(x) = \sum_{\dim(X) = 0} x^{2 p(X)} + \sum_{\dim(X) = 1} (x^2 + 1) x^{2 p(X)} + \sum_{\dim(X) = 2} (x^4 + 2 x^2 + 1) x^{2 p(X)},
\]
where the summation is taken over all connected components $X$ of $\M^T$.
Substituting the values for $p(X)$ yields the expression of $P_{\M}$ from Theorem \ref{theorem_2}.
The final statement about the Hodge numbers follows, as in the case of $\N(3, 4, 3)$, from (\ref{3.3.4}).


\appendix

\section{Singular programs I}
\label{appendix_A}

{\small
\begin{verbatim}
ring r=0,(x,y,z),dp;
proc weight-decomposition(list u,list v)
{list the_variables=(x,y,z); list w=list(); list g=list();
 int i,j,k,e;
 for (i=1; i<=size(u); i=i+1) {for (j=1; j<=size(v); j=j+1)
 {for (k=1; k<=3; k=k+1){w=w+list(-v[j]-u[i]+the_variables[k]);};};};
 for (i=1; i<=size(u); i=i+1) {for (j=1; j<=size(u); j=j+1)
                                           {g=g+list(u[i]-u[j]);};};
 for (i=1; i<=size(v); i=i+1) {for (j=1; j<=size(v); j=j+1)
                                           {g=g+list(v[i]-v[j]);};};
 g=delete(g,size(g));
 for (i=1; i<=size(g); i=i+1) {e=1; for (j=1; j<=size(w); j=j+1)
  {if (w[j]==g[i] and e==1) {w=delete(w,j); e=0;};};};
 return(w);};

proc positive_part(list v)
{int i; int p; p=0; for(i=1;i<=12;i=i+1) {if(v[i]>0) {p=p+1;};};
 return(p);};

proc the_values(list w, list l)
{int i; list v; v=list(); for(i=1; i<=12; i=i+1)
 {v=v+list((w[i]/x)*l[1]+(w[i]/y)*l[2]+(w[i]/z)*l[3]);};
 return(v);};

proc positive-parts(list w)
{list d; d=list(); int i; list omega;
 omega=list(list(0,1,4),list(1,4,0),list(4,0,1), list(1,0,4),
  list(4,1,0), list(0,4,1)); for(i=1; i<=6; i=i+1)
 {d=d+list(positive_part(the_values(w,omega[i])))};
 return(d);};
\end{verbatim}
}


\section{Singular programs II}
\label{appendix_B}

{\small
\begin{verbatim}
ring r=0,(x,y,z),dp;

int i,j; poly P, q, q1, q2, l1, l2, l; P=0;
int points, lines; points = 0; lines = 0;

list s1, s2, s3, s4, s5, d;

s1=list(x,y,z);
s2=list(2x, 2y, 2z, x+y, x+z, y+z);
s3=list(3x, 3y, 3z, 2x+y, 2x+z, x+2y, 2y+z, x+2z, y+2z, x+y+z);
s4=list(4x, 4y, 4z, 3x+y, 2x+2y, x+3y, 3x+z, 2x+2z, x+3z,
3y+z, 2y+2z, y+3z, 2x+y+z, x+2y+z, x+y+2z);
s5=list(5x,5y,5z,4x+y,3x+2y,2x+3y,x+4y,4x+z,3x+2z,2x+3z,x+4z,
4y+z,3y+2z,2y+3z,y+4z,3x+y+z,x+3y+z,x+y+3z,
2x+2y+z,2x+y+2z,x+2y+2z);

proc add(poly p, list l)
{int i; list ll; ll=list();
for(i=1; i<=size(l); i=i+1){ll=ll+list(p+l[i]);};
return(ll);};

proc positive_part(list l)
{int i; int p; p=0; for (i=1; i<=size(l); i=i+1) {if (l[i]>0) {p=p+1;};};
return(p);};

proc values(list w, list l)
{int i; list v; v=list(); for (i=1; i<=size(w); i=i+1)
{v=v+list((w[i]/x)*l[1]+(w[i]/y)*l[2]+(w[i]/z)*l[3]);};
return(v);};

proc sub(list l, list ll)
{list lll; int i,j,e; lll=l; for (j=1; j<=size(ll); j=j+1)
{e=1; for(i=1; i<=size(lll); i=i+1)
{if (lll[i]==ll[j] and e==1) {lll=delete(lll,i); e=0;};};};
return(lll);};

proc id0(poly a, poly b)
{return(sub(s5, (sub(s5,
         add(a+x, s3)+add(a+y, s3)+add(b+x, s2)+add(b+y, s2)))));};

proc id2(list l)
{list ll; ll = list(); int i; for (i=1; i<=size(l); i=i+1)
{ll = ll+ add(l[i], s3);};
return(sub(s5, (sub(s5, ll))));};

proc id3(list l)
{list ll; ll = list(); int i; for (i=1; i<=size(l); i=i+1)
{ll = ll+ add(l[i], s2);};
return(sub(s5, (sub(s5, ll))));};

proc point_2(list l)
{points=points+2;
return(x^(2*positive_part(values(l, list(0,1,8))))
+x^(2*positive_part(values(l, list(1,0, 8)))));};

proc point_3(list l)
{points=points+3;
return(x^(2*positive_part(values(l, list(0,1,8))))
+x^(2*positive_part(values(l, list(8,1,0))))
+x^(2*positive_part(values(l, list(0,8,1)))));};

proc point_3_1(list l)
{points=points+3;
return(x^(2*positive_part(values(l, list(0,1,8))))
+x^(2*positive_part(values(l, list(1,0,8))))
+x^(2*positive_part(values(l, list(8,1,0)))));};

proc point_3_2(list l)
{points=points+3;
return(x^(2*positive_part(values(l, list(0,1,8))))
+x^(2*positive_part(values(l, list(1,8,0))))
+x^(2*positive_part(values(l, list(8,0,1)))));};

proc point_6(list l)
{points=points+6;
return(x^(2*positive_part(values(l, list(0,1,8))))
+x^(2*positive_part(values(l, list(1,0,8))))
+x^(2*positive_part(values(l, list(8,1,0))))
+x^(2*positive_part(values(l, list(1,8,0))))
+x^(2*positive_part(values(l, list(0,8,1))))
+x^(2*positive_part(values(l, list(8,0,1)))));};

proc line_3(list l)
{lines=lines+3;
return((1+x^2)*x^(2*positive_part(values(l, list(0,1,8))))
+(1+x^2)*x^(2*positive_part(values(l, list(8,1,0))))
+(1+x^2)*x^(2*positive_part(values(l, list(0,8,1)))));};

proc line_3_1(list l)
{lines=lines+3;
return((1+x^2)*x^(2*positive_part(values(l, list(0,1,8))))
+(1+x^2)*x^(2*positive_part(values(l, list(1,0,8))))
+(1+x^2)*x^(2*positive_part(values(l, list(8,1,0)))));};

proc line_3_2(list l)
{lines=lines+3;
return((1+x^2)*x^(2*positive_part(values(l, list(0,1,8))))
+(1+x^2)*x^(2*positive_part(values(l, list(1,8,0))))
+(1+x^2)*x^(2*positive_part(values(l, list(8,0,1)))));};

proc line_6(list l)
{lines=lines+6;
return((1+x^2)*x^(2*positive_part(values(l, list(0,1,8))))
+(1+x^2)*x^(2*positive_part(values(l, list(1,0,8))))
+(1+x^2)*x^(2*positive_part(values(l, list(8,1,0))))
+(1+x^2)*x^(2*positive_part(values(l, list(1,8,0))))
+(1+x^2)*x^(2*positive_part(values(l, list(0,8,1))))
+(1+x^2)*x^(2*positive_part(values(l, list(8,0,1)))));};

proc surface_3(list l)
{return((1+2*(x^2)+x^4)*x^(2*positive_part(values(l, list(0,1,8))))
+(1+2*(x^2)+x^4)*x^(2*positive_part(values(l, list(1,8,0))))
+(1+2*(x^2)+x^4)*x^(2*positive_part(values(l, list(8,0,1)))));};

proc w0(list u, list v)
{list ll; ll=list(); int i;
for(i=1; i<=4; i=i+1)
{ll=ll+add(-v[1]-u[i], s1)+add(-v[2]-u[i], s1)+add(-v[3]-u[i], s1)
+add(-v[4]-u[i], s2);};
return(ll);};

proc g0(list u, list v)
{list ll; ll= list(); int i,j;
for(i=1; i<=4; i=i+1){for(j=1; j<=4; j=j+1)
{ll=ll+list(u[i]-u[j]);};};
for(i=1; i<=3; i=i+1){for(j=1; j<=3; j=j+1)
{ll=ll+list(-v[i]+v[j]);};};
ll=ll+add(-v[4]+v[1], s1)+add(-v[4]+v[2], s1)+add(-v[4]+v[3], s1);
return(ll);};

proc m0(list u, list v)
{return(sub(w0(u,v), g0(u,v)));};

list alpha;
d=id3(list(x+y+z, y+2z, 2x+z, x+2y));
for(i=1; i<=size(d); i=i+1)
{alpha = m0(list(0, x-z, y-x, z-y), list(z, x, y, d[i]-x-y-z));
P = P + point_2(alpha);};

list beta;
d=id3(list(x+y+z, 2y+z, 2x+z, 2x+y));
for(i=1; i<=size(d); i=i+1)
{beta = m0(list(0, x-y, y-x, z-x), list(y, x, x, d[i]-x-y-z));
P = P + point_6(beta);};

list gamma;
d=id3(list(2y+z, x+y+z, 2x+z, 3y));
for(i=1; i<=size(d); i=i+1)
{gamma = m0(list(0, x-y, y-x, x-2y+z), list(y, x, 2y-x, d[i]-x-y-z));
P = P + point_6(gamma);};

list delta;
d=id3(list(2y+z, 2x+z, x+2y, 2x+y));
for(i=1; i<=size(d); i=i+1)
{delta = m0(list(0, z-y, x-2y+z, 2x-2y), list(y, 2y-z, 2y-x, d[i]-2x-z));
P = P + point_3(delta);};

list epsilon;
d=id3(list(2y+z, x+y+z, 2x+y, 3x));
for(i=1; i<=size(d); i=i+1)
{epsilon = m0(list(0, y-x, x-z, 2x-y-z), list(x, z, -x+y+z, d[i]-2x-y));
P = P + point_6(epsilon);};

list zeta;
d=id3(list(2y+z, x+2y, 2x+y, 3x));
for(i=1; i<=size(d); i=i+1)
{zeta = m0(list(0, x-y, y-x, 2x-y-z), list(y, x, -x+y+z, d[i]-2x-y));
P = P + point_6(zeta);};

list eta;
d=id3(list(3x, 2x+y, x+2y, 3y));
for(i=1; i<=size(d); i=i+1)
{eta = m0(list(0, x-y, y-x, 2y-2x), list(y,x,2x-y, d[i]-x-2y));
P = P + point_3(eta);};

list theta;
d=sub(id3(list(x+y+z, 2x+z, 2x+y, x+2z)), list(x+y+3z, x+2y+2z));
for(i=1; i<=size(d); i=i+1)
{theta = m0(list(0, y-x, z-x, y-z), list(x, x, x-y+z, d[i]-x-y-z));
P = P + point_6(theta);};
d=list(x+y+3z, x+2y+2z);
for(i=1; i<=size(d); i=i+1)
{theta = m0(list(0, y-x, z-x, y-z), list(x, x, x-y+z, d[i]-x-y-z));
P = P + line_6(theta);};

list iota;
d=sub(id3(list(x+2z, x+y+z, 2x+y, x+2y)), list(x+3y+z, x+2y+2z));
for(i=1; i<=size(d); i=i+1)
{iota = m0(list(0, z-y, -x-y+2z, -2y+2z), list(y, x+y-z, x+2y-2z, d[i]-x-2z));
P = P + point_6(iota);};
d=list(x+3y+z, x+2y+2z);
for(i=1; i<=size(d); i=i+1)
{iota = m0(list(0, z-y, -x-y+2z, -2y+2z), list(y, x+y-z, x+2y-2z, d[i]-x-2z));
P = P + line_6(iota);};

list kappa;
d=sub(id3(list(x+2y, 2x+y, 2x+z, x+2z)), list(x+2y+2z));
for(i=1; i<=size(d); i=i+1)
{kappa = m0(list(0, y-x, -x+2y-z, 2y-2z), list(x, x-y+z, x-2y+2z, d[i]-x-2y));
P = P + point_3_2(kappa);};
d=list(x+2y+2z);
for(i=1; i<=size(d); i=i+1)
{kappa = m0(list(0, y-x, -x+2y-z, 2y-2z), list(x, x-y+z, x-2y+2z, d[i]-x-2y));
P = P + line_3_2(kappa);};

list lambda;
d=sub(id3(list(3x, 2x+y, 2x+z, x+y+z)), list(2x+2y+z, 2x+y+2z));
for(i=1; i<=size(d); i=i+1)
{lambda = m0(list(0, x-y, 2x-y-z, x-z), list(y, -x+y+z, z, d[i]-3x));
P = P + line_3_2(lambda);};
d=list(2x+2y+z, 2x+y+2z);
for(i=1; i<=size(d); i=i+1)
{lambda = m0(list(0, x-y, 2x-y-z, x-z), list(y, -x+y+z, z, d[i]-3x));
P = P + surface_3(lambda);};

list mu;
d=sub(id3(list(3x, 2x+y, x+y+z, x+2y)), list(x+3y+z, x+2y+2z));
for(i=1; i<=size(d); i=i+1)
{mu = m0(list(y-x, 0, x-z, x-y), list(x, z, y, d[i]-2x-y));
P = P + point_6(mu);};
d=list(x+3y+z, x+2y+2z);
for(i=1; i<=size(d); i=i+1)
{mu = m0(list(y-x, 0, x-z, x-y), list(x, z, y, d[i]-2x-y));
P = P + line_6(mu);};

list nu;
d=sub(id3(list(x+2y, 2x+y, 3x, 2x+z)), list(2x+2y+z, 2x+y+2z));
for(i=1; i<=size(d); i=i+1)
{nu = m0(list(0, y-x, 2y-2x, -x+2y-z), list(x, 2x-y, 2x-2y+z, d[i]-x-2y));
P = P + point_6(nu);};
d=list(2x+2y+z, 2x+y+2z);
for(i=1; i<=size(d); i=i+1)
{nu = m0(list(0, y-x, 2y-2x, -x+2y-z), list(x, 2x-y, 2x-2y+z, d[i]-x-2y));
P = P + line_6(nu);};

list xi;
d=sub(id3(list(2x+y, x+2y, x+y+z)),
                   list(3x+y+z, x+3y+z, x+y+3z, x+2y+2z, 2x+y+2z));
for(i=1; i<=size(d); i=i+1)
{xi = m0(list(0, x-y, x-z, x-z), list(y, z, -x+y+z, d[i]-2x-y));
P = P + point_3_2(xi);};
d=list(3x+y+z, x+3y+z, x+y+3z, x+2y+2z, 2x+y+2z);
for(i=1; i<=size(d); i=i+1)
{xi = m0(list(0, x-y, x-z, x-z), list(y, z, -x+y+z, d[i]-2x-y));
P = P + line_3_2(xi);};

proc w1(list u, list v)
{return(add(-v[1]-u[1], s3)+add(-v[1]-u[2], s2)+add(-v[2]-u[1], s3)+
             add(-v[2]-u[2], s2));};

proc g1(list u, list v)
{return(list(0,0,0,v[1]-v[2],v[2]-v[1]) + add(u[2]-u[1], s1));};

proc m1(list u, list v)
{return(sub(w1(u, v), g1(u, v)) + list(u[1]+v[1]-x-y-z, u[1]+v[2]-x-y-z));};

list omicron;
q1=2x; q2=2y;
d=sub(id2(list(q1,q2)), list(3x+y+z, x+3y+z));
for(i=1; i<=size(d); i=i+1)
{omicron=m1(list(d[i]-q1-q2, 0), list(q1, q2));
P = P + point_3(omicron);};

q1=2x; q2=y+z;
d=sub(id2(list(q1,q2)), list(3x+y+z, x+2y+2z));
for(i=1; i<=size(d); i=i+1)
{omicron=m1(list(d[i]-q1-q2, 0), list(q1, q2));
P = P + point_3_1(omicron);};

q1=2x; q2=x+y;
d=sub(id2(list(q1,q2)), list(2x+3y,2x+y+2z,2x+2y+z,3x+y+z));
for(i=1; i<=size(d); i=i+1)
{omicron=m1(list(d[i]-q1-q2, 0), list(q1, q2));
P = P + point_6(omicron);};
d=list(2x+3y,2x+y+2z);
for(i=1; i<=size(d); i=i+1)
{omicron=m1(list(d[i]-q1-q2, 0), list(q1, q2));
P = P + line_6(omicron);};

q1=x+z; q2=y+z;
d=sub(id2(list(q1,q2)), list(3x+y+z,x+3y+z,2x+2y+z,2x+y+2z,x+2y+2z));
for(i=1; i<=size(d); i=i+1)
{omicron=m1(list(d[i]-q1-q2, 0), list(q1, q2));
P = P + point_3(omicron);};
d=list(3x+y+z,x+3y+z,2x+2y+z);
for(i=1; i<=size(d); i=i+1)
{omicron=m1(list(d[i]-q1-q2, 0), list(q1, q2));
P = P + line_3(omicron);};

proc w2(list u, list v)
{return(add(-v[1]-u[1], s2)+add(-v[2]-u[1], s3)+add(-v[3]-u[1], s3)
 +add(-v[1]-u[2], s1)+add(-v[2]-u[2], s2)+add(-v[3]-u[2], s2)
 +add(-v[2]-u[3], s1)+add(-v[3]-u[3], s1));};

proc g2(list u, list v)
{return(sub(list(0, 0, 0, 0, 0, v[2]-v[3], v[3]-v[2])
 +add(u[2]-u[1], s1)+add(u[3]-u[1], s2)+add(u[3]-u[2], s1)
 +add(v[1]-v[2], s1)+add(v[1]-v[3], s1), list(v[1]-v[2]+x)));};

proc m2(list u, list v)
{return(sub(w2(u, v), g2(u, v))+list(u[1]+v[2]-x-y-z, 
                                  u[1]+v[3]-x-y-z, -v[1]-u[3]));};

list pi;
l=y; q=2z;
d=sub(id0(l, q), list(2x+2y+z, x+3y+z, x+2y+2z));
for(i=1; i<=size(d); i=i+1)
{pi=m2(list(d[i]-l-x-y, d[i]-q-x-y, 0), list(-d[i]+q+l+x+y, x, y));
P = P + point_6(pi);};

l=y; q=2x;
d=sub(id0(l, q), list(2x+y+2z, 2x+2y+z, x+3y+z, 3x+2y));
for(i=1; i<=size(d); i=i+1)
{pi=m2(list(d[i]-l-x-y, d[i]-q-x-y, 0), list(-d[i]+q+l+x+y, x, y));
P = P + point_6(pi);};
d=list(2x+y+2z);
for(i=1; i<=size(d); i=i+1)
{pi=m2(list(d[i]-l-x-y, d[i]-q-x-y, 0), list(-d[i]+q+l+x+y, x, y));
P = P + line_6(pi);};

l=y; q=x+z;
d=sub(id0(l, q), list(x+y+3z, 2x+2y+z, x+3y+z));
for(i=1; i<=size(d); i=i+1)
{pi=m2(list(d[i]-l-x-y, d[i]-q-x-y, 0), list(-d[i]+q+l+x+y, x, y));
P = P + point_6(pi);};
d=list(x+y+3z);
for(i=1; i<=size(d); i=i+1)
{pi=m2(list(d[i]-l-x-y, d[i]-q-x-y, 0), list(-d[i]+q+l+x+y, x, y));
P = P + line_6(pi);};

l=z; q=2x;
d=sub(id0(l, q), list(2x+y+2z, x+2y+2z, 3x+y+z));
for(i=1; i<=size(d); i=i+1)
{pi=m2(list(d[i]-l-x-y, d[i]-q-x-y, 0), list(-d[i]+q+l+x+y, x, y));
P = P + point_6(pi);};

l=z; q=x+y;
d=sub(id0(l, q), list(2x+y+2z, x+2y+2z, 2x+2y+z));
for(i=1; i<=size(d); i=i+1)
{pi=m2(list(d[i]-l-x-y, d[i]-q-x-y, 0), list(-d[i]+q+l+x+y, x, y));
P = P + point_3(pi);};

proc w3(list u, list v)
{return(add(-v[1]-u[1], s1)+add(-v[1]-u[2], s1)
 +add(-v[2]-u[1], s4)+add(-v[2]-u[2], s4));};

proc g3(list u, list v)
{return(list(0, 0, 0, u[1]-u[2], u[2]-u[1])+add(v[1]-v[2], s3));};

proc m3(list u, list v)
{return(sub(w3(u, v), g3(u, v))
 +sub(list(u[1]+v[2]-2x-y-z, u[1]+v[2]-x-2y-z, u[1]+v[2]-x-y-2z,
 u[2]+v[2]-2x-y-z, u[2]+v[2]-x-2y-z, u[2]+v[2]-x-y-2z), list(v[2]-x-y-z)));};

list rho;
d=sub(s5, list(5z, x+3y+z, 2x+2y+z, 3x+y+z, x+2y+2z, 2x+y+2z));
for(i=1; i<=size(d); i=i+1)
{rho=sub(w3(list(x, y), list(0, d[i]-x-y)), g3(list(x, y), list(0, d[i]-x-y)))
+list(d[i]-3x-y-z, d[i]-2x-2y-z, d[i]-x-3y-z, d[i]-2x-y-2z, d[i]-x-2y-2z);
P = P + point_3(rho);};
\end{verbatim}
}

\end{document}